\documentclass{amsart}
\usepackage{a4}
\usepackage{graphics}
\usepackage{color}
\usepackage{amssymb}
\usepackage[all]{xy}
\usepackage{amsmath}
\usepackage{stmaryrd}
\usepackage{subfigure}
\usepackage{longtable}
\usepackage{setspace}
\usepackage{braket}
\usepackage{hyperref}
\usepackage{booktabs}
\usepackage{enumitem}
\usepackage{tikz}

\CompileMatrices

\newtheorem{theorem}{Theorem}[section]
\newtheorem{lemma}[theorem]{Lemma}
\newtheorem{proposition}[theorem]{Proposition}
\newtheorem{corollary}[theorem]{Corollary}

\theoremstyle{definition}

\newtheorem{definition}[theorem]{Definition}

\newtheorem{construction}[theorem]{Construction}

\newtheorem{remark}[theorem]{Remark}
\theoremstyle{remark}

\numberwithin{equation}{section}
\def\quot{/\!\!/}

\def\rq#1{\widehat{#1}}

\def\b#1{\overline{#1}}
\def\bangle#1{\langle #1 \rangle}

\def\CC{{\mathbb C}}
\def\KK{{\mathbb K}}

\def\ZZ{{\mathbb Z}}

\def\QQ{{\mathbb Q}}
\def\PP{{\mathbb P}}

\def\Cox{\mathcal{R}}
\def\canK{\mathcal{K}}
\def\tuplen{\overline{n}}

\def\Eff{{\rm Eff}}
\def\Mov{{\rm Mov}}

\def\Ample{{\rm Ample}}
\def\SAmple{{\rm SAmple}}
\def\rlv{{\rm rlv}}

\def\cov{{\rm cov}}
\def\trop{{\rm trop}}

\def\Cl{\operatorname{Cl}}

\def\Spec{{\rm Spec}}

\def\conv{{\rm conv}}
\def\cone{{\rm cone}}
\def\lin{{\rm lin}}

\def\cov{{\rm cov}}

\def\relint{{\rm relint}}
\def\shadow{{\rm sw}}
\def\vol{{\rm vol}}

\begin{document}

\title[On terminal Fano 3-folds with a 2-torus action, part II]%
{On terminal Fano 3-folds with a 2-torus action, \\ part II}

\author[M.~Nicolussi]{Michele Nicolussi} 
\address{Mathematisches Institut, Universit\"at T\"ubingen,
Auf der Morgenstelle 10, 72076 T\"ubingen, Germany}

\begin{abstract}
We continue the classification of terminal Fano
threefolds with an effective two-torus action. 
In earlier work we settled the $\QQ$-factorial case with Picard number one. 
Here we treat the larger class of varieties that do not admit
any contraction of a prime divisor; these are called
combinatorially minimal.
\end{abstract}

\maketitle


\section{Introduction}

This article contributes to the classification of Fano varieties,
meaning normal projective varieties with an ample anticanonical divisor.
We work over an algebraically closed field $\KK$ of characteristic zero.

From the point of view of the Minimal Model Program,
the most important subclass of Fano varieties are those having
terminal $\QQ$-factorial singularities.
In dimension two, these are just the smooth del Pezzo surfaces.
In dimension three, we know the list of smooth Fano threefolds
due to Iskovskikh~\cite{Is1,Is2} and Mori/Mukai~\cite{MoMu} but
the singular case is still wide open.
Retricting to toric geometry, Kasprzyk~\cite{Ka} classified
all terminal toric Fano threefolds by means of lattice polytopes,
following an idea of Borisov and Borisov~\cite{BoBo}.
Their combinatorial approach can be outlined in geometrical terms as follows:
\begin{enumerate}
\item
find all terminal $\QQ$-factorial toric Fano threefolds with Picard number one;
\item
determine the maximal possible rank for the divisor class group of terminal toric Fano threefolds
that do not admit any contraction of a prime divisor;
\item
find all such varieties;
\item
``grow'' the varieties found in (i) and (iii) to obtain all terminal toric Fano threefolds.
\end{enumerate}

This paper contributes to the generalization of this approach
to varieties~$X$ of complexity one, i.e.
coming with an effective action of a torus $T$ such that $\dim(T)=\dim(X)-1$ holds.
The foundation was laid in a joint work by the author
together with Bechtold, Hausen and Huggenberger: in~\cite{BeHaHuNi},
a polyhedral complex called the \emph{anticanonical complex}
is associated to any rational Fano variety $X$ of complexity one.
Properties of the singularities (such as terminality) are characterized by
the lattice points inside of the anticanonical complex.
In the same article, item (i) of the approach was generalized:
the anticanonical complex was used to obtain
the full list of terminal $\QQ$-factorial Fano threefolds $X$
of complexity one with Picard number $\rho(X)=1$.

The first main result of this paper is a generalization of
item (ii) for varieties of complexity one.
We denote with $\delta(X)$ the rank of the divisor class group $\Cl(X)$ of a variety $X$.
One always has $\rho(X)\le\delta(X)$,
since the Picard group is a subgroup of $\Cl(X)$.
We call a variety \emph{combinatorially minimal}
if it does not admit any contraction of a prime divisor.

\begin{theorem}
\label{thm:combmin-bound}
Let $X$ be a non-toric rational combinatorially minimal terminal Fano
threefold with an effective two-torus action.
Then $\rho(X)\le\delta(X)\le3$ holds.
\end{theorem}

This Theorem allows a methodical approach to the classification of
combinatorially minimal terminal Fano threefolds of complexity one,
i.e. the analog of item~(iii).
Our second result is the classification of these varieties in the $\QQ$-factorial case
by means of their Cox rings;
see~\cite[Sec.~1.4]{ArDeHaLa} for a precise formulation of Cox rings.

\begin{theorem}
\label{thm:combmin-list}
Let $X$ be a non-toric rational combinatorially minimal terminal $\QQ$-factorial Fano
threefold with an effective two-torus action
and divisor class group of rank bigger than one.
Then the Cox ring $\Cox(X)$ of $X$ appears in the list below.
The $\Cl(X)$-degrees of the generators $T_1,\ldots,T_r$
are denoted as columns $w_i\in\Cl(X)$ of a matrix $[w_1,\ldots,w_r]$.
\begin{center}
\begin{tabular}[h]{llcc}
\toprule
No.
&
$\Cox(X)$
&
$\Cl(X)$
&
$[w_1,\ldots, w_r]$
\\
\midrule
2.01
&
$\KK[T_1,\ldots,T_6]/\bangle{T_1T_2+T_3T_4+T_5T_6}$
&
$\ZZ^2$
&
$
\left[
\begin{smallmatrix}
1 & 0 & 1 & 0 & 1 & 0 \\
0 & 1 & 0 & 1 & 0 & 1
\end{smallmatrix}
\right]
$
\\
\midrule
2.02
&
$\KK[T_1,\ldots,T_6]/\bangle{T_1T_2^2+T_3T_4^2+T_5T_6^2}$
&
$\ZZ^2$
&
$
\left[
\begin{smallmatrix}
1 & 0 & 1 & 0 & 1 & 0 \\
0 & 1 & 0 & 1 & 0 & 1
\end{smallmatrix}
\right]
$
\\
\midrule
2.03
&
$\KK[T_1,\ldots,T_6]/\bangle{T_1T_2+T_3T_4+T_5T_6}$
&
$\ZZ^2\oplus\ZZ/3\ZZ$
&
$
\left[
\begin{smallmatrix}
1 & 0 & 1 & 0 & 1 & 0 \\
0 & 1 & 0 & 1 & 0 & 1 \\
\overline{2} & \overline{1} & \overline{1} & \overline{2} & \overline{0} & \overline{0}
\end{smallmatrix}
\right]
$
\\
\midrule
2.04
&
$\KK[T_1,\ldots,T_6]/\bangle{T_1T_2^2+T_3T_4^2+T_5T_6^2}$
&
$\ZZ^2\oplus\ZZ/3\ZZ$
&
$
\left[
\begin{smallmatrix}
1 & 0 & 1 & 0 & 1 & 0 \\
0 & 1 & 0 & 1 & 0 & 1 \\
\overline{1} & \overline{1} & \overline{2} & \overline{2} & \overline{0} & \overline{0}
\end{smallmatrix}
\right]
$
\\
\midrule
2.05
&
$\KK[T_1,\ldots,T_6]/\bangle{T_1T_2+T_3^2T_4+T_5T_6}$
&
$\ZZ^2$
&
$
\left[
\begin{smallmatrix}
2 & 0 & 1 & 0 & 1 & 1 \\
0 & 1 & 0 & 1 & 0 & 1
\end{smallmatrix}
\right]
$
\\
\midrule
2.06
&
$\KK[T_1,\ldots,T_6]/\bangle{T_1T_2+T_3T_4^2+T_5^2T_6^2}$
&
$\ZZ^2$
&
$
\left[
\begin{smallmatrix}
1 & 1 & 2 & 0 & 1 & 0 \\
1 & 1 & 0 & 1 & 0 & 1
\end{smallmatrix}
\right]
$
\\
\midrule
2.07
&
$\KK[T_1,\ldots,T_6]/\bangle{T_1T_2T_3+T_4T_5+T_6^2}$
&
$\ZZ^2$
&
$
\left[
\begin{smallmatrix}
1 & 1 & 0 & 2 & 0 & 1 \\
0 & 0 & 2 & 1 & 1 & 1
\end{smallmatrix}
\right]
$
\\
\midrule
2.08
&
$\KK[T_1,\ldots,T_6]/\bangle{T_1T_2T_3+T_4T_5^2+T_6^2}$
&
$\ZZ^2$
&
$
\left[
\begin{smallmatrix}
1 & 1 & 0 & 2 & 0 & 1 \\
1 & 0 & 1 & 0 & 1 & 1
\end{smallmatrix}
\right]
$
\\
\midrule
2.09
&
$\KK[T_1,\ldots,T_6]/\bangle{T_1T_2T_3^2+T_4T_5+T_6^2}$
&
$\ZZ^2$
&
$
\left[
\begin{smallmatrix}
1 & 1 & 0 & 2 & 0 & 1 \\
0 & 0 & 1 & 1 & 1 & 1
\end{smallmatrix}
\right]
$
\\
\midrule
2.10
&
$\KK[T_1,\ldots,T_6]/\bangle{T_1T_2T_3^2+T_4T_5^2+T_6^2}$
&
$\ZZ^2$
&
$
\left[
\begin{smallmatrix}
1 & 1 & 0 & 2 & 0 & 1 \\
0 & 0 & 1 & 0 & 1 & 1
\end{smallmatrix}
\right]
$
\\
\midrule
2.11
&
$\KK[T_1,\ldots,T_6]/\bangle{T_1T_2+T_3^2T_4+T_5^2}$
&
$\ZZ^2$
&
$
\left[
\begin{smallmatrix}
2 & 0 & 1 & 0 & 1 & 1 \\
1 & 1 & 0 & 2 & 1 & 0
\end{smallmatrix}
\right]
$
\\
\midrule
2.12
&
$\KK[T_1,\ldots,T_6]/\bangle{T_1T_2T_3+T_4^2+T_5^2}$
&
$\ZZ^2\oplus\ZZ/2\ZZ$
&
$
\left[
\begin{smallmatrix}
1 & 1 & 0 & 1 & 1 & 0 \\
0 & 0 & 2 & 1 & 1 & 1 \\
\overline{0} & \overline{1} & \overline{1} & \overline{1} & \overline{0} & \overline{0}
\end{smallmatrix}
\right]
$
\\
\bottomrule
\end{tabular}
\end{center}
Any two of the Cox rings listed in the table correspond to non-isomorphic varieties.
Only No.~2.01 and No.~2.02 are smooth.
\end{theorem}

In Section~\ref{sec:tvarcpl1} we recall the description of Fano varieties of complexity one
via certain defining matrices.
Moreover we review definition and basic properties of the anticanonical complex.
In Section~\ref{sec:combmin-rho} we focus on combinatorially minimal varieties
and derive useful properties on their defining data.
Section~\ref{sec:combmin-3folds} is devoted to the proof of Theorem~\ref{thm:combmin-bound}.
In Section~\ref{sec:Qfact-classif} we prove Theorem~\ref{thm:combmin-list}.
In Section~\ref{sec:growanticancomp} we discuss the generalization of item~(iv) of the approach,
setting the basis for a future complete classification of terminal Fano threefolds of complexity one.
The author would like to thank J\"urgen Hausen for the constant help
and numerous fruitful discussions.

\tableofcontents


\section{$T$-varieties with a torus action of complexity one}
\label{sec:tvarcpl1}

For convenience, we gather here notation and basic facts
on rational projective varieties $X$
that come with an effective torus action of complexity one;
the original references are~\cite{HaSu,HaHe,BeHaHuNi}.
The Cox rings $\Cox(X)$ of these varieties are
obtained in the following way.

\begin{construction}
\label{constr:RAPdown}
Fix $r \in \ZZ_{\ge 1}$ and a sequence 
$\tuplen=(n_0, \ldots, n_r) \in \ZZ_{\ge 1}^{r+1}$, set 
$n := n_0 + \ldots + n_r$, and fix  
integers $m \in \ZZ_{\ge 0}$ and $0 < s < n+m-r$.
A pair $(A,P)$ of \emph{defining matrices}
consists of
\begin{itemize}
\item 
a matrix $A := [a_0, \ldots, a_r]$ of size $2\times(r+1)$
with pairwise linearly independent column vectors from $\KK^2$,
\item 
an integral block matrix $P$ of size $(r + s) \times (n + m)$, 
whose columns are pairwise different primitive
vectors generating $\QQ^{r+s}$ as a cone.
\begin{eqnarray*}
P
& = & 
\left[ 
\begin{array}{cc}
L & 0 
\\
d & d'  
\end{array}
\right],
\end{eqnarray*}
where $d$ is an $(s \times n)$-matrix, $d'$ an $(s \times m)$-matrix 
and $L$ is an $(r \times n)$-matrix built from tuples 
$l_i := (l_{i1}, \ldots, l_{in_i}) \in \ZZ_{\ge 1}^{n_i}$ 
as follows
\begin{eqnarray*}
L
& = & 
\left[
\begin{array}{cccc}
-l_0 & l_1 &   \ldots & 0 
\\
\vdots & \vdots   & \ddots & \vdots
\\
-l_0 & 0 &\ldots  & l_{r} 
\end{array}
\right].
\end{eqnarray*}
\end{itemize}
We denote with $v_{ij}$ and $v_{k}$ the columns of the matrix $P$,
where $0 \le i \le r$, $1 \le j \le n_i$ and $1 \le k \le m$.
Over the same indices define the variables $T_{ij}$ and $S_k$ and
consider the polynomial ring $\KK[T_{ij},S_k]$.
For every $0 \le i \le r$, define a monomial
\begin{eqnarray*}
T_i^{l_i} 
&  := &
T_{i1}^{l_{i1}} \cdots T_{in_i}^{l_{in_i}}.
\end{eqnarray*}
Denote by $\mathfrak{I}$ the set of 
all triples $I = (i_1,i_2,i_3)$ with 
$0 \le i_1 < i_2 < i_3 \le r$ 
and define for any $I \in \mathfrak{I}$ 
a trinomial 
$$
g_I \ := \ g_{i_1,i_2,i_3} \ := \ \det
\left[
\begin{array}{ccc}
T_{i_1}^{l_{i_1}} & T_{i_2}^{l_{i_2}} & T_{i_3}^{l_{i_3}}
\\
a_{i_1} & a_{i_2} & a_{i_3}
\end{array}
\right].
$$
Let $P^*$ denote the transpose of $P$,
consider the factor group 
$K := \ZZ^{n+m}/\rm{im}(P^*)$
and the projection $Q \colon \ZZ^{n+m} \to K$.
We define a $K$-grading on $\KK[T_{ij},S_k]$ via
$$ 
\deg(T_{ij}) 
 \ := \ 
Q(e_{ij}),
\qquad
\deg(S_{k}) 
 \ := \ 
Q(e_{k}).
$$
Then the trinomials $g_I$ are $K$-homogeneous,
all of the same degree $\mu \in K$.
In particular, we obtain a $K$-graded factor ring  
\begin{eqnarray*}
R(A,P)
& := &
\KK[T_{ij},S_k; \; 0 \le i \le r, \, 1 \le j \le n_i, 1 \le k \le m] 
 / \
\bangle{g_I; \; I \in \mathfrak{I}}.
\end{eqnarray*}
\end{construction}

\begin{remark}
\label{rem:ci}
The ring $R(A,P)$ of Construction~\ref{constr:RAPdown}
is a complete intersection: with $g_i := g_{i,i+1,i+2}$ 
we have 
$$ 
\bangle{g_I; \; I \in \mathfrak{I}}
\ = \ 
\bangle{g_0,\ldots,g_{r-2}},
\qquad\quad
\dim(R(A,P)) \ = \ n+m-(r-1).
$$
We can always assume that $P$ is \emph{irredundant} 
in the sense that $l_{i1} + \ldots + l_{in_i} \ge 2$
holds for all $i = 0,\ldots, r$.
A redundant matrix $P$ allows the elimination of variables in $R(A,P)$.
\end{remark} 

\begin{remark}
\label{remark:admissibleops}
The following elementary column and row operations 
on the defining matrix $P$ are called \emph{admissible operations}.
They do not change the isomorphy type
of the $K$-graded ring $R(A,P)$;
\begin{enumerate}
\item
swap two columns inside a block
$v_{ij_1}, \ldots, v_{ij_{n_i}}$,
\item
swap two whole column blocks
$v_{ij_1}, \ldots, v_{ij_{n_i}}$
and $v_{i'j_1}, \ldots, v_{i'j_{n_{i'}}}$,
\item
add multiples of the upper $r$ rows
to one of the last $s$ rows,
\item
any elementary row operation among the last $s$
rows,
\item
swap two columns inside the $d'$ block.
\end{enumerate}
\end{remark}

\begin{remark}
\label{rem:fanoRAP}
The \emph{anticanonical class} of the $K$-graded 
ring $R(A,P)$ is 
$$ 
\kappa(A,P)
\ := \ 
\sum_{i,j} Q(e_{ij})
\ + \ 
\sum_k Q(e_k)
\ - \ 
(r-1) \mu
\ \in \ 
K
$$
and the \emph{moving cone} of $R(A,P)$ is
$$ 
\Mov(A,P) 
\ := \ 
\bigcap_{i,j} \cone(Q(e_{uv},e_t; \; (u,v) \ne (i,j))  
\ \cap \ 
\bigcap_{k} \cone(Q(e_{uv},e_{t}; \; t \ne k)  
\ \subseteq \
K_\QQ.
$$
The $K$-graded ring $R(A,P)$ is the Cox ring of 
a Fano variety if and only if $\kappa(A,P)$ lies in
the relative interior of $\Mov(A,P)$.
This Fano variety can be constructed from $R(A,P)$ as follows.
\end{remark}

\begin{construction}
\label{constr:RAPFano}
Consider the $K$-graded ring $R(A,P)$ 
of Construction~\ref{constr:RAPdown}
and assume that $\kappa(A,P)$ lies in the 
relative interior of $\Mov(A,P)$.
Then the $K$-grading on $\KK[T_{ij},S_k]$
defines an action of the quasitorus 
$H := \Spec \; \KK[K]$ on $\b{Z} := \KK^{n+m}$
leaving $\b{X} := V(g_I; I\in\mathfrak{I}) \subseteq \b{Z}$ 
invariant. Consider 
$$ 
\rq{Z}_c 
\ := \ 
\{z \in \b{Z}; \; 
f(z) \ne 0 \text{ for some } 
f \in \KK[T_{ij},S_k]_{ \nu \kappa(A,P)}, \, 
\nu \in \ZZ_{>0} \}
\ \subseteq \ 
\b{Z},
$$
the set of $H$-semistable points with 
respect to the weight $\kappa(A,P)$.
Then $\rq{X} := \b{X} \cap \rq{Z}_c$ is 
an open $H$-invariant set in $\b{X}$ 
and we have a commutative diagram
$$ 
\xymatrix{
{\rq{X}}
\ar[r]
\ar[d]^{\quot H}_{\pi}
&
{\rq{Z}_c}
\ar[d]^{\quot H}_{\pi}
\\
X(A,P)
\ar[r]
&
Z_c}
$$
where $X(A,P)$ is a Fano variety with torus action of complexity one,
$Z_c := \rq{Z}_c \quot H$ is a complete toric variety,
the downward maps $\pi$ are characteristic spaces and
the lower horizontal arrow is a closed embedding. 

Moreover we define $Z$ as the minimal toric subvariety of $Z_c$
containing $X$ as a closed subvariety.
The fan $\Sigma$ of $Z$ lives in $\QQ^{r+s}$.
The primitive generators of the rays of $\Sigma$
are precisely the columns of the matrix $P$.

We have
$$ 
\dim(X(A,P)) = s+1,
\qquad
\Cl(X(A,P)) \ \cong \ K,
$$
$$
-\mathcal{K}_X \ = \ \kappa(A,P),
\qquad
\mathcal{R}(X) \ \cong \ R(A,P).
$$
By the results of~\cite{HaSu,HaHe},
every normal rational Fano variety with an effective torus action
of complexity one arises from this Construction.
\end{construction}


Next, we present the anticanonical complex of a rational Fano variety
$X$ with torus action of complexity one.
This construction was first introduced in~\cite{BeHaHuNi}
in a much more general setting, i.e. that of any
normal Fano variety with a complete intersection Cox ring.

\begin{construction}
Consider the $K$-graded ring $R(A,P)$ of Construction~\ref{constr:RAPdown}
and the corresponding Fano variety $X=X(A,P)$ from Construction~\ref{constr:RAPFano}.
Let $\gamma \subseteq \QQ^{n+m}$ be the positive orthant,
spanned by the canonical basis vectors $e_{ij},e_k \in \ZZ^{n+m}$.
Denote with $B(g_i)$ the Newton polytopes of the relations $g_i$
and let $B := B(g_0) + \ldots + B(g_{r-2})$ be their Minkowski sum.

The \emph{anticanonical polyhedron} of 
$X$ is the dual polyhedron $A_X \subseteq \QQ^{r+s}$ 
of the polytope
$$
B_X \ := \ 
(P^*)^{-1}((Q^{-1}(-\mathcal{K}_X)\cap\gamma) + B - (1,\ldots,1)) 
\ \subseteq \ \QQ^{r+s}.
$$
The \emph{anticanonical complex} of $X$ 
is the coarsest common refinement of polyhedral 
complexes
$$ 
A^c_X \ := \ 
{\rm faces}(A_X) \sqcap \Sigma \sqcap \trop(X).
$$
The \emph{relative interior} of $A_X^c$ 
is the interior of its support 
with respect to the tropical variety $\trop(X)$.
\end{construction}

\begin{theorem}[{\cite[Thm.~1.4]{BeHaHuNi}}]
\label{thm:BHHNmain}
Let $X=X(A,P)$ be Fano.
\begin{enumerate}
\item
$A_X^c$ contains the origin in its relative
interior and all primitive generators of 
the fan $\Sigma$ are vertices of $A_X^c$.
\item 
$X$ has at most log terminal singularities 
if and only if the anticanonical complex 
$A_X^c$ is bounded.
\item
$X$ has at most terminal singularities
if and only if $0$ and the primitive generators 
$v_\varrho$ for $\varrho \in \Sigma^{(1)}$
are the only lattice points of $A_X^c$.
\end{enumerate}
\end{theorem}


Now we turn to the machinery 
developed in~\cite{BeHa,Ha2,ArDeHaLa}. 
Let us briefly summarize the necessary 
notions and statements in a series of remarks 
adapted to our needs.

\begin{remark}
\label{const:rlvu}
Fix defining matrices $(A,P)$ 
and let $\gamma \subseteq \QQ^{n+m}$ 
be the positive orthant,
spanned by the canonical basis vectors 
$e_{ij},e_k \in \ZZ^{n+m}$.
Every face $\gamma_0 \preceq \gamma$ defines a toric orbit 
in $\overline{Z} = \KK^{n+m}$:
$$ 
\overline{Z}(\gamma_0)
\ := \ 
\{
z \in \overline{Z}; \; 
z_{ij} \neq 0 \Leftrightarrow e_{ij} \in \gamma_0  
\text{ and } 
z_{k} \neq 0 \Leftrightarrow e_{k} \in\gamma_0 
\}
\ \subseteq \
\overline{Z}
$$
We say that $\gamma_0 \preceq \gamma$ is an 
\emph{$\mathfrak{F}$-face} for $(A,P)$
if the associated toric orbit meets
the total coordinate space 
$\overline{X} = V(g_I; \; I \in \mathfrak{I}) \subseteq \overline{Z}$,
that means if we have
$$
\overline{X}(\gamma_0) 
\ := \
\overline{X} \cap \overline{Z}(\gamma_0)
\ \ne \
\emptyset.
$$  
In particular, $\overline{X}$ is the disjoint union 
of the locally closed pieces $\overline{X}(\gamma_0)$
associated to the $\mathfrak{F}$-faces.
\end{remark}

\begin{remark}
\label{rem:rlv2}
For the toric variety $Z_c$ and $X=X(A,P)$ of  
Construction~\ref{constr:RAPFano}, we define
the collections of \emph{relevant faces} and the \emph{covering collection}:
\begin{eqnarray*}
\rlv(Z_c) & := & 
\{ \gamma_0 \preceq \gamma; \; u \in Q(\gamma_0)^\circ\},
\\
\rlv(X) & := & 
\{ \gamma_0 \in \rlv(Z_c); \;
\gamma_0 \text{ is an } \mathfrak{F}\text{-face} \},
\\
\cov(X) & := &
\{ \gamma_0\in\rlv(X); \;
\gamma_0 \text{ minimal} \}.
\end{eqnarray*}
Let $\gamma_0^* := \gamma_0^\perp \cap \gamma\preceq \gamma$ 
denote the complementary face of $\gamma_0 \preceq \gamma$.
Then there is a bijection between $\rlv(Z_c)$ and the fan 
$\Sigma_c$ of the toric variety $Z_c$:
$$
\rlv(Z_c) \ \to \ \Sigma_c,
\qquad\qquad
\gamma_0 \ \mapsto \ P(\gamma_0^*).
$$  
The toric orbits of $Z_c$ correspond to the 
cones of the fan $\Sigma_c$ and thus to the 
cones of~$\rlv(Z_c)$. 
Concretely, the toric orbit of $Z_c$ associated 
with $\gamma_0\in \rlv(Z_c)$  is 
$$ 
Z_c(\gamma_0)
\ : = \
\pi(\overline{Z}(\gamma_0)).
$$ 
The relevant faces $\rlv(X)$ of $X$  
define exactly the toric orbits of 
$Z_c$ that intersect $X \subseteq Z_c$
non-trivially
and thus give a locally closed 
decomposition
$$
X 
\ = \ 
\bigcup_{\gamma_0\in\rlv(X)} X(\gamma_0),
\qquad\qquad
X(\gamma_0) 
\ := \ 
X \cap Z_c(\gamma_0) 
\ = \ 
\pi((\overline{X}(\gamma_0)).
$$
The fan $\Sigma$, defining the minimal toric open subset 
$Z \subseteq Z_c$ that contains $X$, is generated by the cones
$\sigma=P(\gamma_0^*)$, where $\gamma_0\in\rlv(X)$.
\end{remark}

\begin{remark}
\label{rem:divcones}
Let $X=X(A,P)$ arise from Construction~\ref{constr:RAPFano}.
Then the cones of effective, movable, semiample
and ample divisor classes are given as
$$ 
\Eff(X) 
\ = \ 
Q(\gamma),
\qquad
\Mov(X) 
\ = \ 
\Mov(A,P)
\ = 
\bigcap_{\gamma_0 \text{ facet of }\gamma} Q(\gamma_0),
$$
$$
\SAmple(X)
\ =
\bigcap_{\gamma_0\in\rlv(X)} Q(\gamma_0),
\qquad
\Ample(X)
\ =
\bigcap_{\gamma_0\in\rlv(X)} Q(\gamma_0)^\circ.
$$
In particular, the GIT-fan of the $H$-action 
on $\overline{X}$ induces the Mori chamber decomposition,
meaning that it subdivides $\Mov(X)$ into 
the nef cones of the small birational relatives of~$X$. 
\end{remark}

\begin{remark}
\label{rem:Qfact}
Let $X=X(A,P)$ arise from Construction~\ref{constr:RAPFano}.
For any $\gamma_0\in\rlv(X)$ and $x\in X(\gamma_0)$,
the following statements hold:
\begin{enumerate}
\item
$x$ is $\QQ$-factorial if and only 
if $Q(\gamma_0)$ is full-dimensional,
\item
$x$ is factorial if and only if 
$Q$ maps $\lin(\gamma_0)\cap\ZZ^{n+m}$ onto $\Cl(X)$,
\item
$x$ is smooth if and only if $x$ is factorial 
and all $z \in \pi^{-1}(x)$ are
smooth in~$\overline{X}$.
\end{enumerate}
\end{remark}

As shown in~\cite[Corollary~1.5.16]{Nico},
we can compute the dimension of a stratum
$X(\gamma_0)$ in the following way:

\begin{proposition}
\label{prop:dimstrata}
Let $X=X(A,P)$ be $\QQ$-factorial.
Then for any $\gamma_0\in\rlv(X)$ we have
$$
\dim X(\gamma_0) \ = \ 
\max\{ k\in\ZZ_{\ge0}; \, \exists \textrm{ chain } \gamma_k \prec \ldots \prec \gamma_0 \textrm{ in } \rlv(X) \}.
$$
\end{proposition}


In the following we give an explicit description of the vertices of
the anticanonical complex of a Fano variety of complexity one.
The original source for the $\QQ$-factorial case is~\cite[Section~4]{BeHaHuNi},
whereas in~\cite[Section~3]{ABHW} Arzhantsev, Braun, Hausen and Wrobel
generalized the description to the non-$\QQ$-factorial case.

\begin{construction}
In the situation of Construction~\ref{constr:RAPFano}, define
the linear space $\lambda := \{0\} \times \QQ^s \subseteq \QQ^{r+s}$, 
the canonical basis vectors $e_1,\ldots, e_r$ and
$e_0 := -e_1- \ldots -e_r$.
Then the tropical variety of $X$ is given by
$$ 
\trop(X) 
\ = \ 
\lambda_0 \cup \ldots \cup \lambda_r
\ \subseteq \ 
\QQ^{r+s},
\qquad \text{where} \quad
\lambda_i \ := \ \cone(e_i) + \lambda.
$$
A cone $\tau \in \Sigma$ is called
\begin{itemize}
\item
\emph{big}\index{cone!big},
if $\tau \cap \relint(\lambda_i) \ne \emptyset$ 
holds for each $i = 0, \ldots, r$;
\item
\emph{elementary big}\index{cone!elementary big}
if it is big, has no rays inside $\lambda$
and precisely one inside each $\lambda_i$;
\item
a \emph{leaf cone}\index{cone!leaf}
if $\tau \subseteq \lambda_i$ holds for some $i$. 
\end{itemize}
Moreover we call \emph{$P$-elementary cone}
any cone $\sigma$ of the form
$$
\sigma \ = \ \cone(v_{0j_0},\ldots,v_{rj_r}) \subset \QQ^{r+s},
$$
which means that $\sigma$ is defined by precisely one $v_{ij}$ for every $i=0,\ldots,r$.
Note that a $P$-elementary cone does not need to be a cone of $\Sigma$.
In the $\QQ$-factorial case, elementary cones and $P$-elemntary cones coincide
but this is not true in the general case.

Let $\sigma \in \Sigma$ be a $P$-elementary cone.
We assign the following positive integers
to the rays $\varrho = \cone(v_{ij}) \in \sigma^{(1)}$
of $\sigma$ and to $\sigma$ itself:
$$
l_\varrho \ := \ l_{ij},
\qquad
\ell_{\sigma,\varrho} 
\ := \ 
l_{{\varrho}}^{-1} \prod_{\varrho' \in \sigma^{(1)}}{l_{\varrho'}},
\qquad 
\ell_{\sigma} \ 
:= \ 
\sum_{\varrho \in \sigma^{(1)}}{\ell_{\sigma,\varrho}} 
- 
(r-1)\prod_{\varrho \in \sigma^{(1)}}{l_{\varrho}}.
$$
Moreover, in $\QQ^{r+s}$, we define vectors:
$$
v_{\sigma} 
\ := \ 
\sum_{\varrho \in \sigma^{(1)}}{\ell_{\sigma,\varrho} v_{\varrho}},
\qquad\qquad
v'_{\sigma}
\ := \ 
\ell_{\sigma}^{-1}v_{\sigma}.
$$
Let $A_X^c$ be the anticanonical complex of $X = X(A,P)$.
The \emph{lineality part} of $A_X^c$ is the polyhedral 
complex $A_{X,0}^c := A_X^c \sqcap \lambda$.
The \emph{$i$-th leaf} of $A_X^c$ is the polyhedral 
complex $A_X^c \sqcap \lambda_i$.
According to~\cite[Proposition~3.8]{ABHW},
if $X$ is log terminal and $\QQ$-factorial, then we have
\begin{eqnarray*}
\vert A_X^c \sqcap \lambda \vert
& = &
\conv(v_{\varrho}, v_{\sigma}'; \ 
\varrho \in \Sigma \text{ with } \varrho \subseteq \lambda,
\ \sigma \text { is $P$-elementary}),
\\
\vert A_X^c \sqcap \lambda_i \vert
& = &
\conv(v_{\varrho}, v_{\sigma}'; \ 
\varrho \in \Sigma \text{ with } \varrho \subseteq \lambda_i,
\ \sigma \text { is $P$-elementary}).
\end{eqnarray*} 
\end{construction}

\begin{remark}[{\cite[Theorem~3.13]{ABHW}}]
\label{rem:logterm2lbound}
Let $X = X(A,P)$ be a Fano variety
and consider a $P$-elementary cone
$\sigma = \varrho_0 + \ldots + \varrho_r \in \Sigma$
defining a log terminal singularity. 
Assume $l_{\varrho_0} \ge  \ldots \ge l_{\varrho_r}$.  
Then $l_{\varrho_3} =  \ldots = l_{\varrho_r} = 1$
holds and $(l_{\varrho_0},l_{\varrho_1},l_{\varrho_2})$ 
is a \emph{platonic triple}, i.e.~one of
$$ 
(l_{\varrho_0},l_{\varrho_1},1),
\qquad
(l_{\varrho_0},2,2),
\qquad
(3,3,2),
\qquad
(4,3,2),
\qquad
(5,3,2).
$$
\end{remark}

\begin{proposition}
\label{prop:triple1xy}
Let $X=X(A,P)$ be a Fano variety and consider a $P$-elementary cone
$\sigma = \varrho_0 + \ldots + \varrho_r \in \Sigma$
defining a terminal singularity.
Then at most two $l_{\varrho_i}$ differ from one.
\end{proposition}

\begin{proof}
Assume $l_{\varrho_0}\ge \ldots\ge l_{\varrho_r}$.
By Remark~\ref{rem:logterm2lbound} the 
triple $(l_{\varrho_0},l_{\varrho_1},l_{\varrho_2})$ is
platonic and $l_{\varrho_3}= \ldots = l_{\varrho_r}=1$ holds.
The denominator $\ell_\sigma$ of the entries of $v_{\sigma}'$
is determined by the platonic triple. In particular for
$$
(3,3,2), \qquad
(4,3,2), \qquad
(5,3,2)
$$
we obtain $\ell_{\sigma}=3,2,1$ respectively.
Moreover, the numerator is a sum in which every addend is a
multiple of all $l_{\varrho_i}$ except one.
Hence the vertex $v_{\sigma}'$ turns out to be a lattice point,
contradicting terminality by Theorem~\ref{thm:BHHNmain}.
The same happens for $(l_{\varrho_0},2,2)$ if $l_{\varrho_0}$
is even, so assume that $l_{\varrho_0}\ge 3$ is odd.
We show that this case has an integral point
on $A_{X}^c$ and therefore does not define a terminal singularity.
The primitive lattice points $v_{\varrho_i}$ have coordinates
$$
v_{\varrho_0} = (-l_{\varrho_0},\ldots,-l_{\varrho_0},d_{01},\ldots,d_{0s})^T,
\qquad
v_{\varrho_i} = (0,\ldots,0,l_{\varrho_i},0,\ldots,0,d_{i1},\ldots,d_{is})^T
$$
for all $i=1,\ldots,r$, where $l_{\varrho_i}$ occupies the $i$-th entry of $v_{\varrho_i}$.
Let $u:=v_\sigma'$ be the vertex of $A_{X,0}^c$ defined by $\sigma$.
It has coordinates $u_j=0$ for $j=1,\dots,r$ and
$$
u_{r+t} = d_{0t} + \frac{l_{\varrho_0}}{2}d_{1t} + \frac{l_{\varrho_0}}{2}d_{2t}
+ l_{\varrho_0}\sum_{k=3}^r d_{kt}
\qquad\forall\, t=1,\ldots,s \,.
$$
Using the fact that $l_{\varrho_0}$ is odd,
we see that on the edge connecting $v_{\varrho_0}$ with $u$
lies at least one lattice point, namely
$$
\frac{l_{\varrho_0}-1}{l_{\varrho_0}}u + \frac{1}{l_{\varrho_0}}v_{\varrho_0} = 
(-1,\ldots,-1, q_1,\dots,q_s)^T, \quad\text{with}
$$
$$
q_t = d_{0t} + \frac{l_{\varrho_0}-1}{2}d_{1t} + \frac{l_{\varrho_0}-1}{2}d_{2t}
+ (l_{\varrho_0}-1)\sum_{k=3}^r d_{kt}
\qquad\forall\, t=1,\ldots,s \,.
$$
Hence the platonic triple $(l_{\varrho_0},l_{\varrho_1},l_{\varrho_2})$
is of the first type, i.e. $l_{\varrho_2}=1$.
\end{proof}

lastly we study the effect of terminality on the strata of threefolds,
more precisely on the corresponding subsets of weights.

\begin{remark}
\label{rem:posstrata}
According to Remark~\ref{rem:Qfact}~(ii),
the stratum $X(\gamma_0)\subset X$ consists
of factorial points of $X=X(A,P)$ if and only if
$Q(\lin(\gamma_0)\cap\ZZ^{n+m})=\Cl(X)$ holds.
In dimension three, terminal singularities occur as isolated points,
see e.g.~\cite[Cor.~4.6.6]{MatsukiMMP}.
According to Proposition~\ref{prop:dimstrata},
every $\gamma_0\in\rlv(X)\setminus\cov(X)$
defines a stratum $X(\gamma_0)$ of positive dimension.
This has to be smooth, in particular factorial.
Therefore its weights generate $\Cl(X)$ as an abelian group.
\end{remark}


\section{Combinatorial minimality}
\label{sec:combmin-rho}

In this Section, we derive bounds on the Picard 
number of combinatorially minimal, $\QQ$-factorial, 
log terminal Fano varieties with torus action 
of complexity one. 
We begin by recalling the necessary notions
and facts from~\cite{HuKe, Ha2}.

\medskip
A \emph{small quasimodification}\index{small quasimodification}
is a birational map $X \dasharrow X'$ 
of complete varieties that restricts to 
a regular isomorphism $U \to U'$ between 
open subsets $U \subseteq X$ and $U' \subseteq X'$ 
having complements of codimension at least two 
in $X$ and $X'$ respectively.
We say that $X$ is \emph{combinatorially minimal} if it 
has no contractible prime divisors in the sense that
any birational map $X \dasharrow X'$ which is defined in 
codimension two is a small quasimodification.

\begin{remark}
\label{rem:combmin}
A projective variety $X$ with finitely generated Cox ring
is combinatorially minimal if and only if
its cone of movable divisor classes 
coincides with its cone of effective divisor classes. 
In particular for a variety $X = X(A,P)$ of complexity one,
this precisely means that every extremal ray of the effective cone 
$\Eff(X) \subseteq \Cl_{\QQ}(X)$ hosts the degrees of 
at least two of the generators $T_{ij}, S_{k}$ of the 
Cox ring $\mathcal{R}(X) = R(A,P)$, see~\cite{Ha2}.
\end{remark}

First we give a characterization for the defining matrix $P$
of a combinatorially minimal variety $X=X(A,P)$.
Consider two mutually dual exact sequences of
finite dimensional rational vector spaces
$$
\begin{tikzpicture}
\node (Z1) {$0$};
\node (Z2) [below of=Z1, node distance=1cm] {$0$};
\node (L) [right of=Z1, node distance=1.5cm] {$L_\QQ$};
\node (K) [below of=L, node distance=1cm] {$K_\QQ$};
\node (F) [right of=L, node distance=2cm] {$F_\QQ$};
\node (E) [below of=F, node distance=1cm] {$E_\QQ$};
\node (N) [right of=F, node distance=2cm] {$N_\QQ$};
\node (M) [below of=N, node distance=1cm] {$M_\QQ$};
\node (Z4) [right of=N, node distance=1.5cm] {$0$};
\node (Z3) [below of=Z4, node distance=1cm] {$0$};
\draw[->] (Z1) to (L);
\draw[->] (L) to (F);
\draw[->] (F) to (N);
\draw[->] (N) to (Z4);
\draw[->] (K) to (Z2);
\draw[->] (E) to node[below] {{\footnotesize $Q$}} (K);
\draw[->] (M) to (E);
\draw[->] (Z3) to (M);
\end{tikzpicture}
$$
Denote by $(f_1,\ldots,f_r)$ a basis for $F_\QQ$
and by $(e_1,\ldots,e_r)$ its dual basis for $E_\QQ$.
Moreover define $\delta:=\cone(f_1,\ldots,f_r)$
and $\gamma:=\cone(e_1,\ldots,e_r)$.
We call an element $e\in E_\QQ$ an \emph{$L_\QQ$-invariant separating linear form}
\index{invariant separating linear form}
for $\delta_1,\delta_2 \preceq \delta$ if
$$
e|_{L_\QQ} = 0, \quad e|_{\delta_1}=0, \quad e|_{\delta_2}=0, \quad 
\delta_1 \cap e^\perp = \delta_2 \cap e^\perp = \delta_1 \cap \delta_2 .
$$

\begin{lemma}[{Invariant Separation Lemma~\cite[2.2.3.2]{ArDeHaLa}}]
\label{lemma:ISL}
Consider $\delta_1,\delta_2\preceq\delta$ and
their corresponding faces $\gamma_i := \delta_i^\perp\cap\gamma \preceq\gamma$.
Then the following statements are equivalent:
\begin{itemize}
\item
there exists an $L_\QQ$-invariant separating linear form for $\delta_1,\delta_2$;
\item
$Q(\gamma_1)^\circ \cap Q(\gamma_2)^\circ \neq \emptyset$.
\end{itemize}
\end{lemma}

Define the weights $w_{ij} := \deg(T_{ij})$ and $w_k := \deg(S_k)$.
We call a variable $T_{ij},S_k$ \emph{extremal}\index{extremal variable}
if its weight $w_{ij},w_k \in \Cl_{\QQ}(X)$ 
sits on an extremal ray of the effective cone 
$\Eff(X) \subseteq \Cl_{\QQ}(X)$.
Moreover, we call a weight $w \in \{w_{ij},w_k\}$
\emph{exceptional}\index{exceptional weight},
if $\QQ_{\ge0}w$ is an extremal ray of $\Eff(X)$ and
no other weight lies on $\QQ_{\ge0}w$.
All variables with exceptional weights
are extremal, but the converse does not hold.
As a matter of fact, by Remark~\ref{rem:combmin},
$X$ is combinatorially minimal if and only if no weight is exceptional.

\medskip
Only in this Section, we rename the weights $w_1,\dots,w_{n+m}$
and let $v_1,\dots,v_{n+m}$ be the corresponding columns of $P$.

\begin{lemma}
\label{lem:excepweight}
The weight $w\in\{w_{ij},w_k\}$ is exceptional if and only if
$\QQ^{r+s}$ is already generated as a cone by all columns of P
except the one that corresponds to $w$.
\end{lemma}

\begin{proof}
An exceptional weight $w_i$ is characterized by the fact that
$\cone(w_j ; j\neq i)$ and $\cone(w_i)$ allow a separating linear form.
By the Invariant Separation Lemma~\ref{lemma:ISL}, this is equivalent to the fact
that the corresponding Gale dual cones intersect in their relative interiors, i.e.
$$
\cone(v_j ; j\ne i)^\circ \cap \cone(v_i)^\circ \neq \emptyset .
$$
This is the case if and only if $v_i\in\cone(v_j ; j\ne i)^\circ$ holds.
Since the cone over all columns is $\QQ^{r+s}$,
the last condition is equivalent to $\cone(v_j ; j\ne i) = \QQ^{r+s}$.
\end{proof}

\begin{proposition}
\label{prop:combmin-char}
The variety $X=X(A,P)$ is combinatorially minimal if and only if
for every column $v_\varrho$ of $P$ the following holds:
$$
\cone(v_{\varrho'}; \  \varrho'\in\Sigma^{(1)}\setminus\{\varrho\}) \ \neq \ \QQ^{r+s}.
$$
\end{proposition}

\begin{proof}
The assertion follows directly from Lemma~\ref{lem:excepweight}.
\end{proof}

Given a matrix $P$ and an index-set $I\subset\ZZ_{\ge1}$,
we denote with $P_I$ the matrix obtained from $P$
by deleting the $i$-th column, for all $i\in I$.

\begin{corollary}
\label{cor:combmin-char}
Consider the variety $X=X(A,P)$ with rank of $\Cl(X)$ being $\delta(X)\geq2$.
Assume that $\cone(w_i)=\cone(w_j)$ holds
for two weights $w_i,w_j$, with $i\neq j$.
Then $\det(P_{I}) = 0$ holds for every $I \subset\{1,\ldots,n+m\}$
such that $i,j\in I$ and $\vert I \vert = \delta(X)$.
\end{corollary}

\begin{proof}
By construction we have $n+m=r+s+\delta(X)$. Define
\begin{align*}
\tau_i :=& \cone(v_1,\ldots,\widehat{v_i},\ldots,v_{n+m}), \\
\tau_j :=& \cone(v_1,\ldots,\widehat{v_j},\ldots,v_{n+m}).
\end{align*}
According to Lemma~\ref{lem:excepweight} neither of them equals $\QQ^{r+s}$.
Moreover, by the Invariant Separation Lemma, there is a linear form separating them,
i.e. $\tau_i \cap \tau_j$ is a proper face of both cones.
This means that $\tau_i \cap \tau_j =
\cone(v_1,\ldots,\widehat{v_i},\ldots,\widehat{v_j},\ldots,v_{n+m})$
is not full-dimensional.
In particular any collection of $r+s$ columns of $P$ that
does not contain $v_i$ nor $v_j$ is linearly dependent.
\end{proof}

In the sequel, let $\alpha$ denote the difference 
between the number of extremal rays of the 
effective cone $\Eff(X)$ and the rank $\delta(X)$
of the divisor class group $\Cl(X)$.
Broadly speaking, $\alpha$ measures how far
$\Eff(X)$ is from being simplicial.

\begin{lemma}
\label{lem:picbound1}
Let $X = X(A,P)$ be combinatorially minimal. 
Then one of the following holds:
\begin{itemize}
\item
$\dim(X) \ge \delta(X)$ and $m\ge 2\delta(X)-2$;
\item
$\dim(X) \ge \alpha + 2 + m/2$ and  $m< 2\delta(X)-2$.
\end{itemize}
\end{lemma}

\begin{proof}
The effective cone $\Eff(X) \subseteq \Cl_\QQ(X)$ is 
of full dimension and  has 
$\delta(X) + \alpha$ vertices.
Since $X$ is combinatorially minimal, 
the number $n+m$ of variables $T_{ij}$, $S_{k}$ is 
bounded from below by $n+m \ge 2\delta(X) + 2\alpha$.
If the number $m$ of variables $S_k$ satisfies
$m \ge 2 \delta(X) - 2$, then the assertion 
follows from
$$ 
\dim(X)
\ = \ 
n + m + 2 - (r+1) - \delta(X)
\ \ge \ 
m + 2 - \delta(X)
\ \ge \ 
\delta(X).
$$
So, consider the case $m < 2 \delta(X) - 2$.
There are at least two extremal variables of type
$T_{ij}$ having their weights on different rays
and we have in total at least 
$2 \delta(X) + 2 \alpha - m$ 
extremal variables of type $T_{ij}$.
For each of these  $T_{ij}$, 
we must have $n_i \ge 2$. This gives
$$ 
n-(r+1) \ = \ (n_0-1) + \ldots + (n_r-1)
\ \ge \ 
\frac{2 \delta(X) + 2 \alpha - m}{2}.
$$
The assertion then follows from 
\begin{align*}
\dim(X) \ &= \ 
n + m + 2 - (r+1) - \delta(X) \\
&\ge \ \frac{2 \delta(X) + 2 \alpha - m}{2}
+ m + 2 - \delta(X) \\
&= \ \alpha + \frac{m}{2} + 2.
\end{align*}
\end{proof}

For $X = X(A,P)$, we denote by $\eta$ the number of extremal 
variables of type $T_{ij}$ and by $\zeta = n - \eta$ 
the number of non-extremal variables of type $T_{ij}$.

\begin{lemma}
\label{lem:picbound2}
Let $X = X(A,P)$ be combinatorially minimal.
Then we have
$$
\delta(X)
\ \le \ 
\dim(X) + r - 1 - \zeta - 2 \alpha.
$$
\end{lemma}

\begin{proof}
Just observe that $\dim(X) + \delta(X) + r-1
= \eta + \zeta + m \ge2 \delta(X) + 2 \alpha + \zeta$
holds.
\end{proof}

The above estimate is useful for large $\zeta$, for example 
$\zeta \ge r-2$. For small $\zeta$, we need statements concerning
the cases $\alpha = 0$ and $\alpha = 1$. 
Observe that the following three Lemmas,
including the estimate for the case $\alpha=0$,
do not require combinatorial minimality.

\begin{lemma}
\label{lem:exbigcone}
Let $X=X(A,P)$ be log terminal and Fano.
Then $m<\dim(X)$ implies the existence of a big cone in $\Sigma$.
\end{lemma}

\begin{proof}
The lineality part $A_{X,0}^c$ has precisely $m+b$ vertices,
where $b$ is the number of $P$-elementary cones.
Moreover it is full-dimensional in the lineality space $\lambda$
of $\trop(X)$, meaning that it has dimension $s$.
This is only possible with at least $s+1$ vertices,
hence the assertion follows.
\end{proof}

\begin{lemma}
\label{lem:lterm2bound}
Let $X=X(A,P)$ be $\QQ$-factorial, log terminal and Fano.
If $m<\dim(X)$ holds, then we have
$$ 
\frac{n+m}{2}
\ \le \ 
\dim(X) + \delta(X).
$$
\end{lemma}

\begin{proof}
Lemma~\ref{lem:exbigcone} ensures
the existence of an elementary big cone.
Therefore we can apply~\cite[Corollary~4.8]{BeHaHuNi}
and obtain that $r-1$, the number of relations,
equals at most $\dim(X) + \delta(X)$.
The assertion follows from
$$
n+ m 
\ = \ 
\dim(X) + \delta(X) + r - 1
\ \le \ 
2(\dim(X)+\delta(X)).
$$ 
\end{proof}

\begin{lemma}
\label{lem:degineff}
For any $X=X(A,P)$, the degree $\mu$ of the relations $g_{I}$ of $R(A,P)$
lies in the effective cone $\Eff(X) \subseteq \Cl(X)_{\QQ}$.
Moreover $\mu$ lies in its interior
if and only if $0 \notin \cone(v_{k_1},\ldots,v_{k_t})^\circ$
holds for any collection $1\le k_1 < \ldots < k_t \le m$.
\end{lemma}

\begin{proof}
The first assertion is trivial.
The degree $\mu$ lies on a facet $\tau$ of $\Eff(X)$
if and only if all weights $w_{ij}$ lie on $\tau$ as well.
By the Invariant Separation Lemma,
this is precisely the case, when there exist $v_{k_1},\ldots,v_{k_t}$
such that $0 \in \cone(v_{k_1},\ldots,v_{k_t})^\circ$.
\end{proof}

\begin{lemma}
\label{lem:picbound3}
Let $X = X(A,P)$ be a log terminal 
Fano variety, such that $\mu\in\Eff(X)^\circ$ holds.
If $\alpha = 0$, $\zeta \le r-2$ and $m < \dim(X)$ hold, 
then we have
$$
\delta(X)
\ \le \ 
\frac{2}{r-1-\zeta}
\dim(X)  - \frac{m+\zeta}{r-1 - \zeta}.
$$
\end{lemma}

\begin{proof}
In the relations of $R(A,P)$, consider the monomials 
$T_{i}^{l_{i}}$ consisting only of extremal variables $T_{ij}$.
Since $\mu\in\Eff(X)^\circ$ and $\alpha=0$ hold,
each such monomial comprises at least $\delta(X)$ variables.
There are at least $r+1 - \zeta$ such monomials.
We obtain
$$ 
n
\ = \ 
\eta + \zeta
\ \ge \ 
(r-1 - \zeta ) \delta(X)
+ 2 \delta(X)
+ \zeta.
$$
Since $X$ is log terminal with $m < \dim(X)$, 
we can apply Lemma~\ref{lem:lterm2bound}
and see that $2 \dim(X) + 2 \delta(X)$ is bigger 
or equal to $n+m$. 
Combining the two estimates gives the assertion.
\end{proof}

For the case $\alpha = 1$, we use geometrical properties
of $d$-dimensional polyhedral cones with $d+1$ extremal rays.
Here we gather and prove the relevant facts.

\begin{lemma}
\label{lem:cone-d+1}
Let $d\ge3$ and $\sigma \subseteq \QQ^d$ be a pointed convex 
polyhedral $d$-dimensional cone with $d+1$ extremal
rays.
Let $v_1, \ldots, v_{d+1} \in \QQ^d$ be primitive generators 
of the extremal rays of $\sigma$ and $w_1,\ldots, w_{d+1} \in \QQ$
the Gale dual configuration.
Set $D \ := \ \{1, \ldots, d+1\}$ and
$$
D_{-} \ := \ \{i \in D; \; w_{i} < 0\},
\quad
D_{0} \ := \ \{i \in D; \; w_{i} = 0\},
\quad
D_{+} \ := \ \{i \in D; \; w_{i} > 0\}.
$$
Moreover, for any subset $I \subseteq D$, denote by 
$I^{c} \subseteq D$ its complement and define cones 
$\sigma_I := \cone(v_i; \;  i \in I) \subseteq \QQ^d$ and 
$\tau_I := \cone(w_i; \; i \in I^{c}) \subseteq \QQ$.
Then the following statements hold:
\begin{enumerate}[label=(\roman*)]
\item
$\sigma_I$ is a proper face of $\sigma$ if and only 
if $\tau_I = \{0\}$ or $\tau_I = \QQ$ holds.
\item
There are at least two $i$ with  $w_i > 0$ and at least 
two $j$ with $w_j < 0$.
\item
We have $\sigma_I^\circ \subseteq \sigma^\circ$ if and only if 
$D_{-} \cup D_{0} \subseteq I$ or $D_{+} \cup D_{0} \subseteq I$ 
holds.
\item
We have $\sigma_I^\circ \cap \sigma_{J}^\circ \ne \emptyset$ if and only if 
$D_{-} \subseteq I$, $D_{+} \subseteq J$, $I \cap D_{0} = J \cap D_{0}$
or  
$D_{+} \subseteq I$, $D_{-} \subseteq J$, $I \cap D_{0} = J \cap D_{0}$
holds.
\end{enumerate}
In particular, $\sigma_{-} := \cone(v_{i}; \, i \in D_{-} \cup D_{0})$
and  $\sigma_{+} := \cone(v_{i}; \, i \in D_{+} \cup D_{0})$ form
the unique pair of minimal cones satisfying 
$\sigma_\pm^\circ \subseteq \sigma^\circ$,
$\cone(\sigma_-,\sigma_+)=\sigma$ and 
$\sigma_-^\circ \cap \sigma_+^\circ \ne \emptyset$.
\end{lemma}

\begin{proof}
Let $P \colon \QQ^{d+1} \to \QQ^d$ be the linear
map sending $e_i$ to $v_i$ and $Q \colon \QQ^{d+1} \to \QQ$
the one sending $e_i$ to $w_i$.
For $I$ consider 
$\delta_I := \cone(e_i; \; i \in I) \subseteq \QQ^{d+1}$ 
and 
$\gamma_I := \cone(e_i; \;  i \in I^{c}) \subseteq \QQ^{d+1}$.
The Invariant Separation Lemma~\ref{lemma:ISL}
yields for any two $I,J \subseteq \{1,\ldots,d+1\}$ 
the following statements:
\begin{itemize}
\item
There is a $\ker(P)$-invariant separating linear 
form for $\delta_I$ and $\delta_J$
if and only if $\tau_I^\circ \cap \tau_J^\circ \ne \emptyset$ 
holds,
\item
We have $\sigma_I^\circ \cap \sigma_J^\circ \ne \emptyset$ if 
and only if there is a $\ker(Q)$-invariant 
separating linear form for $\gamma_I$ and $\gamma_J$.
\end{itemize}
Observe that $\sigma_{I}, \sigma_{J}$ intersect in a
common face if and only if $\delta_{I}, \delta_{J}$
admit a $\ker(P)$-invariant separating linear form.

Now, assertion~$(i)$ is an immediate consequence of the 
first of the above two items.
Since $\{0\}$ is a face of $\sigma$, we see that 
there must be positive and negative $w_i$.
Assertion~$(ii)$ reflects the fact that every
ray $\sigma_{\{i\}}$ is a face of $\sigma$.
Assertion~$(iii)$ is a special case of~$(iv)$ which in
turn is obtained by adapting the second of the above
items to the setting of the Lemma.
\end{proof}

We are ready to estimate $\delta(X)$ for the 
case $\alpha = 1$ and small $\zeta$.
Again, this statement does not assume combinatorial 
minimality.

\begin{lemma}
\label{lem:picbound4}
Let $X = X(A,P)$ be a log terminal 
Fano variety, such that $\mu\in\Eff(X)^\circ$ holds.
Assume $\alpha = 1$ and $\zeta \le r-2$.
Then we have
$$ 
\delta(X) 
\ \le \ 
\dim(X) + 3 + \zeta - r - m
\ \le \
\dim(X) + 1 - m.
$$
\end{lemma}

\begin{proof}
Let $\omega_-,\omega_+ \subseteq \Eff(X)$ be the 
minimal pair of subcones as in Lemma~\ref{lem:cone-d+1}
and denote by $a_-,a_+$ their respective numbers
of extremal rays. Then we have
$$ 
a_- \ \ge \ 2, \qquad 
a_+ \ \ge \ 2, \qquad 
a_- + a_+ \ \ge \ \delta(X) + 1.
$$
Consider a monomial $T_i^{l_i}$ with only extremal 
variables $T_{ij}$ and let $\omega_{i}$ be the cone 
generated  by the $\deg(T_{ij})$. 
We say that $T_{i}^{l_{i}}$ is of type $(-)$ if 
$\omega_- \subseteq \omega_{i}$ holds and 
of type $(+)$ otherwise. 
Lemma~\ref{lem:cone-d+1}, with $\mu\in\Eff(X)^\circ$,
shows that any $T_{i}^{l_{i}}$ of type $(+)$
satisfies $\omega_+ \subseteq \omega_{i}$.
Let $b_-$ and $b_+$ denote the respective numbers of 
monomials of these types that occur in the $(r-1)$ relations.
Then we have 
$$
b_- + b_+  \ \ge \   r+1 - \zeta,
\qquad 
\eta \ \ge \ b_- a_- + b_+ a_+.
$$ 
For the first estimate, we use that there are at 
least $r+1-\zeta$ monomials
involving only extremal variables.
For the second one, note that every monomial of type
$(\pm)$ has at least $a_{\pm}$ distinct variables.  
We conclude
\begin{eqnarray*}
\dim(X) 
+
\delta(X)
+
r-1
& = &
\eta + \zeta + m
\\ 
& \ge &
b_- a_- + b_+ a_+ + \zeta + m
\\
& \ge &
2(\delta(X)+1) + (b_- -2) a_- + (b_+ -2) a_++ \zeta + m
\\
& \ge &
2(\delta(X)+1)
+ 
2 (r + 1 - \zeta -4) 
+ 
\zeta
+
m.
\end{eqnarray*}
\end{proof}


\section{Combinatorially minimal $3$-folds}
\label{sec:combmin-3folds}

In this Section, we take a closer look at combinatorially 
minimal Fano threefolds with a two-torus action. 
The case with a divisor class group of rank one was settled in~\cite{BeHaHuNi}.
Here the first goal is to find a bound on that rank and then
establish an analogue of~\cite[Lemma~5.2]{BeHaHuNi},
using among other things the estimates from the previous section.
We start by allowing log terminal singularities.

\begin{proposition}[a.k.a.~Theorem~1.1]
\label{prop:picbound}
Let $X = X(A,P)$ be a 3-dimensional, 
log terminal, combinatorially minimal Fano variety.
Then $\delta(X) \le 3$ holds.
\end{proposition}

\begin{proof}
Lemma~\ref{lem:picbound1} tells us that besides 
$\delta(X) \le 3$, we have to consider the 
following two cases:
$$ 
\alpha \ = \ 0 \text{ and } m \ \le \ 2,
\qquad\qquad
\alpha \ = \ 1 \text{ and } m \ = \ 0. 
$$

Firstly assume that $\mu\in\Eff(X)^\circ$ holds.
If $\alpha = 0$ holds, then Lemma~\ref{lem:picbound2}
with $\zeta \ge r-2$ gives $\delta(X) \le 4$ and 
Lemma~\ref{lem:picbound3} with $\zeta \le r-3$ gives
$\delta(X) \le 3$.
If $\alpha = 1$ holds, Lemma~\ref{lem:picbound2}
with $\zeta \ge r-2$ gives $\delta(X) \le 2$ and 
Lemma~\ref{lem:picbound4} with $\zeta \le r-3$ gives
$\delta(X) \le 3$.
So we have to exclude $\delta(X) = 4$,
which only appears in the case $\alpha=0$ with $\zeta \ge r-2$.
Lemma~\ref{lem:picbound2} yields in this case 
$\zeta \le r-2$, thus we have $\zeta = r-2$.
In particular we have a relation involving 
only extremal variables, say the one
with the monomials $T_0^{l_0}, T_1^{l_1}, T_2^{l_2}$.
Together with $\mu\in\Eff(X)^\circ$ and $\alpha=0$
this implies $n_0, n_1, n_2 \ge 4$. 
On the other hand, Lemma~\ref{lem:picbound3} 
gives us $m + \zeta \le 2$ and thus $r \le 4$.
This shows $n + m = 7 + r-1 \le 10$, a contradiction to $n\ge 12$.

Now assume that the degree $\mu$
lies on a facet of the effective cone.
According to Lemma~\ref{lem:degineff}
we have $m\ge2$, therefore $\alpha=0$ and $m=2$.
In particular the weights $w_{ij}$ generate a
$(\delta(X)-1)$-dimensional facet $\theta$ of $\Eff(X)$.
For $r\le\zeta$, Lemma~\ref{lem:picbound2} yields $\delta(X)\le2$
so we can assume $r\ge\zeta+1$.
Since there are at least $r+1-\zeta$ monomials
having only extremal variables and
each of these monomials must have at least $\delta(X)-1$ variables
we conclude
$$
\dim(X) + \delta(X) + r-1 \ = \
\eta + \zeta + m \ \ge \
(r+1-\zeta)(\delta(X)-1) + \zeta + m
$$
and thus $(r-\zeta)(\delta(X)-2)\le \dim(X)-m = 1$.
It follows $\delta(X)\le 3$.
\end{proof}

Next we would like to show that the degree $\mu$
lies in the interior of the effective cone.
By allowing $X=X(A,P)$ to have log terminal singularities, this does not hold.
Consider any log terminal Fano
$\KK^*$-surface $S$ with $\delta(S)=1$,
then $S\times \PP_1$ is a $\QQ$-factorial log terminal Fano
variety of complexity one with Picard number $2$,
such that the degree $\mu$ lies on an extremal ray of its effective cone.
By restricting to the terminal case
we can prove that $\mu\in\Eff(X)^\circ$ holds.

\begin{proposition}
\label{prop:combmin-effint}
Let $X = X(A,P)$ be a non-toric, 3-dimensional, combinatorially minimal,
terminal Fano variety, where the defining matrix $P$ is irredundant.
Then $\mu\in\Eff(X)^\circ$ holds.
\end{proposition}

\begin{definition}
\label{defi:shadow}
Consider a convex set $C\subset \{0\}\times \QQ^{d} \subset \QQ^{d+1}$
and a point $x\in\QQ^{d+1}$ with first coordinate $x_{1} > 0$.
The \emph{shadow}\index{shadow} of $C$ from $x$ is
$$
\shadow(C,x) \ := \ \{ y\in\QQ^{d+1} \ ; \ x\in\conv(y,C) \} \ \subseteq \ \QQ^{d+1}.
$$
Moreover, for any $t\in\QQ$ with $t\ge x_{1}$,
we define the \emph{sliced shadow}\index{shadow!sliced} at height $t$ as
$$
\shadow_t(C,x) \ := \ \{ y\in\shadow(C,x) \ ; \ y_{1}=t \}.
$$
\end{definition}

\begin{lemma}
\label{lem:P2-shadow}
Consider the triangle $C:=\conv(e_2,e_3,-e_2-e_3)\subset\QQ^3$.
Let $x\in\ZZ^3$ be a lattice point with first coordinate $x_1\ge2$.
If the only lattice points of $\conv(C,x)$
are its four vertices and the origin, then $x_1=3$ holds.
\end{lemma}

\begin{proof}
Note that $\conv(C,x)$ always contains its four vertices and the origin.
A further lattice point $y$ lies in it if and only if $x\in\shadow(C,y)$ holds.
Therefore we look for a point $x$ that does not lie
in any $\shadow(C,y)$ for $x\neq y\in\ZZ^3$.

At height $t=2$ all integral points 
are of the form $(2,a,b)$ for some
$a\in\{2u,2u+1\}$ and $b\in\{2v,2v+1\}$,
and they lie in the respective $\shadow_2(C,(1,u,v))$.

For every $t\ge4$, the union of the shadows
$\shadow_t(C,z)$, for all $z\in\ZZ^3$ with $z_1=1$,
contains all integral points at the height $t$,
except for the multiples of $q_1:=(3,3u+1,3v-1)$ and $q_2:=(3,3u-1,3v+1)$.
Since $0\in C$ holds,
these points lie in $\shadow(C,q_1)$ and $\shadow(C,q_2)$, respectively.
Thus the assertion follows.
\end{proof}

Recall that a variety is called \emph{weakly tropical}\index{variety!weakly tropical}
if the fan of its minimal ambient toric variety $Z$ is supported on $\trop(X)$.
This means that there are only leaf cones.

\begin{proof}[Proof~of~Proposition~\ref{prop:combmin-effint}]
By Proposition~\ref{prop:picbound} we only have to consider $\delta(X)\le3$.
The assertion is clear for $\delta(X)=1$,
since $\Eff(X)=\QQ_{\ge0}$ and $\mu>0$ hold.

\smallbreak
Turn to the case $\delta(X)=2$.
Then we have $n+m=r+4$, which implies $m\le3$.
Suppose that $\mu\notin\Eff(X)^\circ$ holds.
Then Lemma~\ref{lem:degineff} yields $m\ge2$
and we have only two possible constellations:
\begin{enumerate}[label=(\alph*)]
\item
$m = 2$, $r = 2$ and $\tuplen=(2,1,1)$,
\item
$m = 3$, $r \ge 2$ and $\tuplen=(1, \ldots, 1)$.
\end{enumerate}
By combinatorial minimality, constellation (a) can only happen if
the weights of the two free variables lie
on one of the two extremal rays of $\Eff(X)$
and all the weights $w_{ij}$ lie on the other extremal ray.
This means that the variety is a product of $\PP_{1}$
with a del Pezzo surface.
By~\cite[Prop. 5.10]{Hug} there are no
non-toric terminal del Pezzo $\KK^*$-surfaces,
hence this case is not compatible with the assumptions.
Constellation (b) allows three weight dispositions:

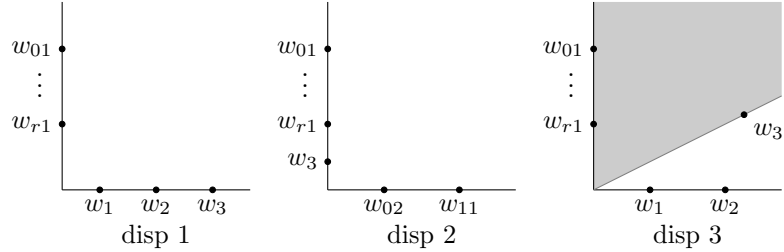
\begin{center}
\begin{tikzpicture}[scale=0.25]
\node at (5,-2.5) {disp $1$};
\draw [thin] (0,0) -- (0,10);
\draw [thin] (0,0) -- (10,0);
\draw [fill] (2,0) circle [radius=0.15];
\node [below] at (2,0) {$w_{1}$};
\draw [fill] (5,0) circle [radius=0.15];
\node [below] at (5,0) {$w_{2}$};
\draw [fill] (8,0) circle [radius=0.15];
\node [below] at (8,0) {$w_{3}$};
\draw [fill] (0,3.5) circle [radius=0.15];
\node [left] at (0,3.5) {$w_{r1}$};
\node [left] at (-0.5,6) {$\vdots$};
\draw [fill] (0,7.5) circle [radius=0.15];
\node [left] at (0,7.5) {$w_{01}$};
\end{tikzpicture}
\
\begin{tikzpicture}[scale=0.25]
\node at (5,-2.5) {disp $2$};
\draw [thin] (0,0) -- (0,10);
\draw [thin] (0,0) -- (10,0);
\draw [fill] (3,0) circle [radius=0.15];
\node [below] at (3,0) {$w_{02}$};
\draw [fill] (7,0) circle [radius=0.15];
\node [below] at (7,0) {$w_{11}$};
\draw [fill] (0,1.5) circle [radius=0.15];
\node [left] at (0,1.5) {$w_{3}$};
\draw [fill] (0,3.5) circle [radius=0.15];
\node [left] at (0,3.5) {$w_{r1}$};
\node [left] at (-0.5,6) {$\vdots$};
\draw [fill] (0,7.5) circle [radius=0.15];
\node [left] at (0,7.5) {$w_{01}$};
\end{tikzpicture}
\
\begin{tikzpicture}[scale=0.25]
\node at (5,-2.5) {disp $3$};
\path[fill=gray!40!] (0,0)--(10,5)--(10,10)--(0,10)--(0,0);
\draw [thin] (0,0) -- (0,10);
\draw [thin] (0,0) -- (10,0);
\draw [gray, thin] (0,0) -- (10,5);
\draw [fill] (3,0) circle [radius=0.15];
\node [below] at (3,0) {$w_{1}$};
\draw [fill] (7,0) circle [radius=0.15];
\node [below] at (7,0) {$w_{2}$};
\draw [fill] (0,3.5) circle [radius=0.15];
\node [left] at (0,3.5) {$w_{r1}$};
\node [left] at (-0.5,6) {$\vdots$};
\draw [fill] (0,7.5) circle [radius=0.15];
\node [left] at (0,7.5) {$w_{01}$};
\draw [fill] (8,4) circle [radius=0.15];
\node [below right] at (8,4) {$w_{3}$};
\end{tikzpicture}
\end{center}

Note that, by almost freeness of the grading,
in all three cases one can assume $\Eff(X)=\QQ^2_{\ge0}$.
The first two dispositions correspond to products of varieties,
hence we rule them out just like before.
Consider disposition $3$.
By Proposition~\ref{prop:triple1xy} and irredundancy of $P$,
there cannot be elementary big cones, hence $X$ is weakly tropical.
Therefore $-\mathcal{K}_X$ lies in $\cone(w_{01},w_{3})$ (possibly on the boundary),
the cone colored in grey in the picture above.
Since $m=3$, we can assume, with admissible operations,
that the last three columns of $P$ are
$(0,\ldots,0,1,0)$, $(0,\ldots,0,0,1)$ and $(0,\ldots,0,-1,-1)$.
These three points are also the vertices of the lineality part $A_{X,0}^c$.
Consider now the $i$-th leaf $A_{X}^c\cap\lambda_i$ of the anticanonical complex,
given as the convex hull of $A_{X,0}^c$ and $v_{i1}$, since $n_i=1$.
Terminality implies that $A_{X}^c\cap\lambda_i$
does not contain additional integral points.
By Lemma~\ref{lem:P2-shadow} we follow $l_{i1}=3$ for all $i=0,\ldots,r$.
This yields $w_{i1}=(0,1)$ and hence $\mu=(0,3)$.
The torsion-free part of the anticanonical class in $K_\QQ$ is
$$
-\mathcal{K}_X \ = \
(r+1) \begin{pmatrix} 0 \\ 1 \end{pmatrix}
+ w_1 + w_2 + w_3 - 
(r-1) \begin{pmatrix} 0 \\ 3 \end{pmatrix} \ = \
\begin{pmatrix} w_1^1+w_2^1+w_3^1 \\ 4-2r+w_3^2 \end{pmatrix}.
$$
Since $r\ge2$, the anticanonical class does not lie
in $\cone(w_{01},w_3)$, a contradiction.

\smallbreak
Lastly consider $\delta(X)=3$.
Assume that $\mu\notin\Eff(X)^\circ$ holds.
Lemma~\ref{lem:degineff} yields $m\ge2$.
We also have the relation $n+m=r+5$, hence $m\le4$.
Note that $m=3$ is excluded by Lemma~\ref{lem:picbound1}.
Therefore we have three constellations:
\begin{enumerate}[label=(\alph*)]
\item
$m = 2$, $r = 2$ and $\tuplen=(2,2,1)$,
\item
$m = 2$, $r = 3$ and $\tuplen=(2,2,1,1)$,
\item
$m = 4$, $r \ge 2$ and $\tuplen=(1, \ldots, 1)$.
\end{enumerate}
We treat both constellations (a) and (b) at once.
By combinatorial minimality the effective cone $\Eff(X)$
is simplicial and the three extremal rays are
$$
\cone(w_{01})=\cone(w_{11}), \qquad
\cone(w_{02})=\cone(w_{12}), \qquad
\cone(w_{1})=\cone(w_{2}).
$$
We apply Remark~\ref{rem:posstrata}
to the relevant face $\gamma_{01,12,1,2}$
and achieve $\Eff(X)=\QQ^3_{\ge0}$.
Thus we are looking at a product of $\PP_1$ with a surface.
This is a contradiction, as already seen before.
In the constellation (c) every monomial consists of only one variable,
hence the respective weights all lie on the same extremal ray of $\Eff(X)$,
while the weights of the free variables
lie on the other two extremal rays, two each.
Since the grading is almost free,
we can assume that $\Eff(X)=\QQ^3_{\ge0}$ holds,
hence $X$ is product of three curves,
again a contradiction to complexity one.
\end{proof}

\begin{lemma}
\label{lem:terminal-combmin}
Let $X = X(A,P)$ be a non-toric, 3-dimensional, combinatorially minimal,
terminal Fano variety, where the matrix $P$ is irredundant
and $\delta(X) > 1$.
Then, after suitable admissible operations,
$P$ fits into one of the following cases:
\begin{enumerate}
\item
We have $\delta(X) = 3$ and 
one of the following constellations:
\begin{enumerate}
\item
$m = 0$, $r = 2$ and $n = 7$, where $\tuplen=(3,3,1)$.
\item
$m = 0$, $r = 3$ and $n = 8$, where $\tuplen=(3,3,1,1)$.
\item
$m = 0$, $r = 3$ and $n = 8$, where $\tuplen=(2,2,2,2)$.
\item
$m = 0$, $r = 4$ and $n = 9$, where $\tuplen=(2,2,2,2,1)$.
\item
$m = 0$, $r = 5$ and $n = 10$, where $\tuplen=(2,2,2,2,1,1)$.
\end{enumerate}
\item
We have $\delta(X) = 2$ and 
one of the following constellations:
\begin{enumerate}
\item
$m = 0$, $r = 2$ and $n = 6$, where $\tuplen=(2,2,2)$.
\item
$m = 0$, $r = 3$ and $n = 7$, where $\tuplen=(2,2,2,1)$.
\item
$m = 0$, $r = 4$ and $n = 8$, where $\tuplen=(2,2,2,1,1)$.
\item
$m = 0$, $r = 2$ and $n = 6$, where $\tuplen=(3,2,1)$.
\item
$m = 0$, $r = 3$ and $n = 7$, where $\tuplen=(3,2,1,1)$.
\item
$m = 0$, $r = 2$ and $n = 6$, where $\tuplen=(4,1,1)$.
\item
$m = 1$, $r = 2$ and $n = 5$, where $\tuplen=(2,2,1)$.
\item
$m = 1$, $r = 3$ and $n = 6$, where $\tuplen=(2,2,1,1)$.
\item
$m = 1$, $r = 2$ and $n = 5$, where $\tuplen=(3,1,1)$.
\item
$m = 2$, $r = 2$ and $n = 4$, where $\tuplen=(2,1,1)$.
\end{enumerate}
\end{enumerate}
\end{lemma}

\begin{proof} 
Note that we have $\delta(X)\le3$ and $\mu\in\Eff(X)^\circ$,
by Propositions~\ref{prop:picbound}
and~\ref{prop:combmin-effint} respectively.

For the case $\delta(X)=3$ we have the relation
\begin{equation}
\label{eq:piczahl3}
n+m = r+5
\end{equation}
coming from $n+m=r+s+\delta(X)$.
Since $n\ge r+1$ always holds, we obtain $m\le 4$.
The case $m=4$ does not fit combinatorial minimality
in the following sense:
according to~\eqref{eq:piczahl3} we would have $n=r+1$,
meaning that every monomial consists only of one variable;
therefore all their weights must lie
in the interior of the effective cone
and consequently the extremal variables are maximal $4$,
contradicting the fact there must be
at least $2\delta(X)$ extremal variables.
Moreover, $m=3$ is excluded
by Lemma~\ref{lem:picbound1}.
Therefore we have $m<\dim(X)$ and there is always a big cone.
This implies $r\le 7$ by~\cite[Corollary~4.8]{BeHaHuNi}.
Combining Proposition~\ref{prop:triple1xy}
with the assumption that $P$ is irredundant we see
that at most two of the $n_i$ equal one.
This leaves us with a finite list of possible configurations.
Many of them can be discharged since
they do not fit combinatorial minimality.
Take for example $m=0$, $r=2$, so that $n=7$.
If we assume $n_0=5$, $n_1=n_2=1$,
then there are not enough extremal variables.
If we assume $n_0=4$, $n_1=2$, $n_2=1$,
there are exactly $6$ variables that may be extremal,
but in this case we would have $\alpha=0$,
which implies that any monomial
comprises at least $\delta(X)$ variables
in order to have combinatorial minimality,
a contradiction to $n_1=2$.
The cases listed in the assertion are
the ones that allow combinatorial minimality.
The case $\delta(X)=2$ is analogous.
\end{proof}


\section{The $\QQ$-factorial classification}
\label{sec:Qfact-classif}

The goal of this Section is the classification of $\QQ$-factorial
combinatorially minimal terminal Fano threefolds $X$ of complexity one.
In order to do that we have to go through the cases of Lemma~\ref{lem:terminal-combmin}
and restrict to $\QQ$-factoriality.
The latter means that the Picard number $\rho(X)$ coincides with the
rank $\delta(X)$ of the divisor class group $\Cl(X)$.
Since the complete proof is quiet long, we refer to~\cite[Section~3.3]{Nico},
where all cases are treated exhaustively.
Here we give a taste of the methods used there by considering the richest case.

\medskip
\noindent
\textbf{Case~$(a)$ of Lemma~\ref{lem:terminal-combmin}~(ii).}

We have $r=2$, $m= 0$, $n = 6$ and $\tuplen=(2,2,2)$.
Combinatorial minimality prescribes
at least two weights on each of the two extremal rays of $\Eff(X)$.
All six weights may be placed on these rays,
therefore we end up with five possible dispositions:
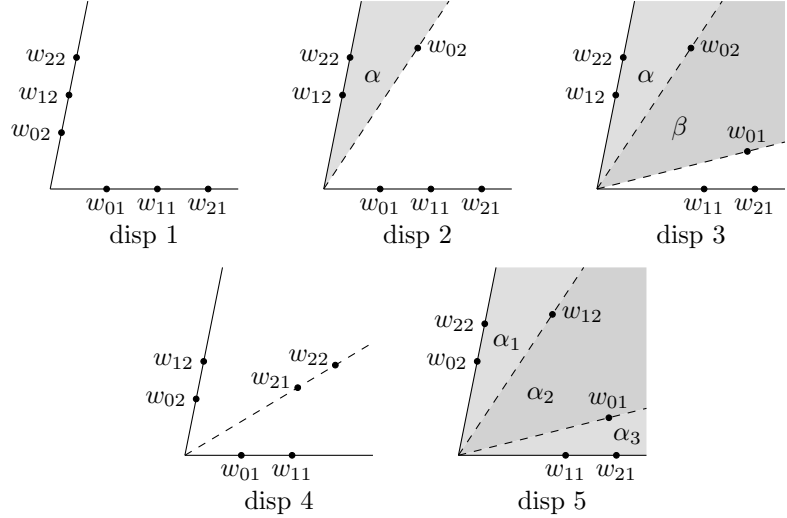
\begin{center}
\begin{tikzpicture}[scale=0.25]
\node at (5,-2.5) {disp $1$};
\draw [thin] (0,0) -- (2,10);
\draw [thin] (0,0) -- (10,0);
\draw [fill] (3,0) circle [radius=0.15];
\node [below] at (3,0) {$w_{01}$};
\draw [fill] (5.7,0) circle [radius=0.15];
\node [below] at (5.7,0) {$w_{11}$};
\draw [fill] (8.4,0) circle [radius=0.15];
\node [below] at (8.4,0) {$w_{21}$};
\draw [fill] (0.6,3) circle [radius=0.15];
\node [left] at (0.6,3) {$w_{02}$};
\draw [fill] (1,5) circle [radius=0.15];
\node [left] at (1,5) {$w_{12}$};
\draw [fill] (1.4,7) circle [radius=0.15];
\node [left] at (1.4,7) {$w_{22}$};
\end{tikzpicture}
\ \ \ \
\begin{tikzpicture}[scale=0.25]
\node at (5,-2.5) {disp $2$};
\path[fill=gray!25!] (0,0)--(6.66,10)--(2,10)--(0,0);
\draw [dashed] (0,0) -- (6.66,10);
\draw [thin] (0,0) -- (2,10);
\draw [thin] (0,0) -- (10,0);
\node at (2.6,6) {$\alpha$};
\draw [fill] (3,0) circle [radius=0.15];
\node [below] at (3,0) {$w_{01}$};
\draw [fill] (5.7,0) circle [radius=0.15];
\node [below] at (5.7,0) {$w_{11}$};
\draw [fill] (8.4,0) circle [radius=0.15];
\node [below] at (8.4,0) {$w_{21}$};
\draw [fill] (5,7.5) circle [radius=0.15];
\node [right] at (5,7.5) {$w_{02}$};
\draw [fill] (1,5) circle [radius=0.15];
\node [left] at (1,5) {$w_{12}$};
\draw [fill] (1.4,7) circle [radius=0.15];
\node [left] at (1.4,7) {$w_{22}$};
\end{tikzpicture}
\ \ \ \
\begin{tikzpicture}[scale=0.25]
\node at (5,-2.5) {disp $3$};
\path[fill=gray!25!] (0,0)--(6.66,10)--(2,10)--(0,0);
\path[fill=gray!35!] (0,0)--(6.66,10)--(10,10)--(10,2.5)--(0,0);
\draw [thin] (0,0) -- (2,10);
\draw [thin] (0,0) -- (10,0);
\draw [dashed] (0,0) -- (6.66,10);
\draw [dashed] (0,0) -- (10,2.5);
\node at (2.6,6) {$\alpha$};
\node at (4.4,3.2) {$\beta$};
\draw [fill] (5.7,0) circle [radius=0.15];
\node [below] at (5.7,0) {$w_{11}$};
\draw [fill] (8.4,0) circle [radius=0.15];
\node [below] at (8.4,0) {$w_{21}$};
\draw [fill] (5,7.5) circle [radius=0.15];
\node [right] at (5,7.5) {$w_{02}$};
\draw [fill] (1,5) circle [radius=0.15];
\node [left] at (1,5) {$w_{12}$};
\draw [fill] (1.4,7) circle [radius=0.15];
\node [left] at (1.4,7) {$w_{22}$};
\draw [fill] (8,2) circle [radius=0.15];
\node [above] at (8,2) {$w_{01}$};
\end{tikzpicture}
\end{center}

\begin{center}
\begin{tikzpicture}[scale=0.25]
\node at (5,-2.5) {disp $4$};
\draw [thin] (0,0) -- (2,10);
\draw [thin] (0,0) -- (10,0);
\draw [dashed] (0,0) -- (10,6);
\draw [fill] (3,0) circle [radius=0.15];
\node [below] at (3,0) {$w_{01}$};
\draw [fill] (5.7,0) circle [radius=0.15];
\node [below] at (5.7,0) {$w_{11}$};
\draw [fill] (0.6,3) circle [radius=0.15];
\node [left] at (0.6,3) {$w_{02}$};
\draw [fill] (1,5) circle [radius=0.15];
\node [left] at (1,5) {$w_{12}$};
\draw [fill] (6,3.6) circle [radius=0.15];
\node [left] at (6.2,4) {$w_{21}$};
\draw [fill] (8,4.8) circle [radius=0.15];
\node [left] at (8.2,5.2) {$w_{22}$};
\end{tikzpicture}
\ \ \ \
\begin{tikzpicture}[scale=0.25]
\node at (5,-2.5) {disp $5$};
\path[fill=gray!25!] (0,0)--(6.66,10)--(2,10)--(0,0);
\path[fill=gray!35!] (0,0)--(6.66,10)--(10,10)--(10,2.5)--(0,0);
\path[fill=gray!25!] (0,0)--(10,2.5)--(10,0)--(0,0);
\draw [thin] (0,0) -- (2,10);
\draw [thin] (0,0) -- (10,0);
\draw [dashed] (0,0) -- (6.66,10);
\draw [dashed] (0,0) -- (10,2.5);
\node at (2.6,6) {$\alpha_1$};
\node at (4.4,3.2) {$\alpha_2$};
\node at (9,1) {$\alpha_3$};
\draw [fill] (5.7,0) circle [radius=0.15];
\node [below] at (5.7,0) {$w_{11}$};
\draw [fill] (8.4,0) circle [radius=0.15];
\node [below] at (8.4,0) {$w_{21}$};
\draw [fill] (5,7.5) circle [radius=0.15];
\node [right] at (5,7.5) {$w_{12}$};
\draw [fill] (1,5) circle [radius=0.15];
\node [left] at (1,5) {$w_{02}$};
\draw [fill] (1.4,7) circle [radius=0.15];
\node [left] at (1.4,7) {$w_{22}$};
\draw [fill] (8,2) circle [radius=0.15];
\node [above] at (8,2) {$w_{01}$};
\end{tikzpicture}
\end{center}

In disposition~2 and~3 define
$\alpha:=\cone(w_{02},w_{12})^\circ$
and $\beta:=\cone(w_{01},w_{02})^\circ$.
In disposition~5 define
$\alpha_1:=\cone(w_{02},w_{12})^\circ$,
$\alpha_2:=\cone(w_{01},w_{12})^\circ$ and
$\alpha_3:=\cone(w_{01},w_{11})^\circ$.

With Proposition~\ref{prop:triple1xy} we obtain
a list of possible exponent configurations:
\begin{enumerate}[label=\Alph*]
\item
$l_0=(1,1)$;
\item
$l_2=(1,1)$;
\item
$l_{01}=l_{11}=l_{21}=1$;
\item
$l_{11}=l_{21}=1$;
\item
$l_{02}=l_{12}=1$;
\item
$l_{01}=l_{21}=1$.
\end{enumerate}
Due to terminality, every disposition allows
only a few of these configurations,
sometimes even just for restricted situations,
depending on the position of the anticanonical class.
The following table summarizes the totality of possible situations:
\begin{center}
\begin{tabular}{c|c|c|c|c|c|c}
 & config A & config B & config C & config D & config E & config F \\
 \hline
 disp $1$ & \checkmark & & \checkmark & & & \\
  \hline
 disp $2$ & \checkmark & \checkmark & \checkmark & \checkmark {\tiny $\alpha$} & & \\
  \hline
 disp $3$ & \checkmark & \checkmark & \checkmark {\tiny $\alpha$,$\beta$} & \checkmark {\tiny $\alpha$} & & \\
  \hline
 disp $4$ & \checkmark & \checkmark & & & \checkmark & \\
  \hline
 disp $5$ & \checkmark & \checkmark & \checkmark {\tiny $\alpha_1$,$\alpha_2$} & & & \checkmark {\tiny $\alpha_1$} \\
\end{tabular}
\end{center}
The combinations of dispositions and configurations that need to be studied
are marked with the sign \checkmark.
A subscript indicates that the anticanonical class $-\canK_X$
has to lie in the given cone(s).

This case provides the first six varieties
of the table of Theorem~\ref{thm:combmin-list},
namely No.~1 and~3 from situation~1A, No.~2 and~4 from~1C,
No.~5 from~2A and No.~6 from~4B.

\medbreak
\emph{Disposition}~1:
since all weights are located on the two extremal rays,
we can assume $\Eff(X)=\QQ^2_{\ge0}$.
For each $w_{ij}$, the two weights $w_{k\ell}$
such that $k\neq i$ and $\ell\neq j$
lie on the other extremal ray. The three together
form a relevant face, to which Remark~\ref{rem:posstrata} applies.
Hence the degree matrix assumes the form
$$
Q \ = \ 
\left[
\begin{array}{cc|cc|cc}
1 & 0 & 1 & 0 & 1 & 0
\\
0 & 1 & 0 & 1 & 0 & 1
\end{array}
\right].
$$

\medbreak
\emph{Situation}~1A: admissible operations,
together with equations from $P\cdot Q^T=0$, yield
$$
P \ = \ 
\left[
\begin{array}{rrrrrc}
-1 & -1 & 1 & 1 & 0 & 0
\\
-1 & -1 & 0 & 0 & 1 & 1
\\
0 & 1 & 0 & d_{112} & 0 & -d_{112}-1
\\
0 & 0 & 0 & d_{212} & 0 & -d_{212}
\end{array}
\right] ,
$$
where we can also assume $0\le d_{112}<d_{212}$.
Therefore we only need to bound $d_{212}$.
For this, take a look at the lineality part $A_{X,0}^c$.
Its vertices are
\begin{align*}
u_1 =& \ \frac{1}{2}( d_{112},d_{212} ),
& u_2 =& \ u_1 \ + \ \left(\frac{1}{2},0\right), \\
u_3 =& \ \frac{1}{2}( -1,0 ),
& u_4 =& \ u_3 \ + \ (1,0), \\
u_5 =& \ \frac{1}{2}( -d_{112}-1,-d_{212} ),
& u_6 =& \ u_5 \ + \ \left(\frac{1}{2},0\right).
\end{align*}
The value $d_{212}$ is odd, otherwise
one between $u_1$ and $u_2$ would be a lattice point,
contradicting terminality.
Since $A_{X,0}^c$ contains no integral point other than the origin,
there are only two possibilities for $(d_{112},d_{212})$, 
namely $(0,1)$ and $(1,3)$.
Both define valid varieties, respectively No.~1 and No.~3.

\medbreak
\emph{Situation}~1C:
here the anticanonical class is $-\mathcal{K}_X=(2,3-l_{02})$.
Since $X$ is a Fano variety and
$\Mov(X)=\QQ_{\ge0}^2$ holds, we have $l_{02}<3$.
From now on we assume $l_{02}=2$, because $l_{02}=1$
has been already discussed in situation~1A.
Admissible operations and $P\cdot Q^T=0$ yield
$$
P \ = \ 
\left[
\begin{array}{rrrrrc}
-1 & -2 & 1 & 2 & 0 & 0
\\
-1 & -2 & 0 & 0 & 1 & 2
\\
0 & 1 & 0 & d_{112} & 0 & -d_{112}-1
\\
0 & 0 & 0 & d_{212} & 0 & -d_{212}
\end{array}
\right] ,
$$
with $0\le d_{112}<d_{212}$.
In order to bound $d_{212}$,
take a look at the lineality part $A_{X,0}^c$.
Its vertices are
\begin{align*}
u_1 =& \ \frac{1}{3}( d_{112},d_{212} ),
& u_2 =& \ \frac{1}{2}( d_{112}+1,d_{212} ), \\
u_3 =& \ \left( -\frac{1}{2},0 \right),
& u_4 =& \ \left( \frac{1}{3},0 \right) \\
u_5 =& \ \frac{1}{3}( -d_{112}-1,-d_{212} ),
& u_6 =& \ \frac{1}{2}( -d_{112},-d_{212} ).
\end{align*}
Consider $C:=\conv(u_3,u_4,u_5)\subset A_{X,0}^c$.
The point $u_5$ lies under the bisection of the third orthant
and, because of the terminality of $X$,
$C$ does not contain integral points. We conclude $d_{212}<20$.
Using the MDSpackage~\cite{MDS} we see that
$(d_{112},d_{212})$ can assume the values $(0,1)$ and $(1,3)$.
These data correspond to varieties No.~2 and~4 respectively.

\medbreak
\emph{Disposition}~2:
we can apply Remark~\ref{rem:posstrata} to the relevant faces
$\gamma_{01,11,22}$, $\gamma_{01,12,21}$
and $\gamma_{01,12,22}$ and obtain
$$
Q \ = \ 
\left[
\begin{array}{cc|cc|cc}
1 & w_{02}^1 & w_{11}^1 & 0 & w_{21}^1 & 0
\\
0 & w_{02}^2 & 0 & 1 & 0 & 1
\end{array}
\right].
$$

\medbreak
\emph{Situation}~2A: homogeneity of the relation
delivers $1+w_{02}^1=l_{11}w_{11}^1=l_{21}w_{21}^1$
and $w_{02}^2=l_{12}=l_{22}$.
We show that the anticanonical class lies in $\cone(w_{11},w_{02})^\circ$.
If we suppose otherwise, then $\gamma_{02,12,22}$ is a relevant face
and in particular $w_{02}^1=1$.
This yields $l_{11},l_{21}\in\{1,2\}$,
but then the anticanonical class does not lie in the prescribed cone.
So $-\mathcal{K}_X\in\cone(w_{11},w_{02})^\circ$ holds,
the face $\gamma_{02,11,21}$ is relevant
and we conclude $l_{12}=1$.
Without loss of generality assume $l_{11}\le l_{21}$.
The requirement $0<\det(-\mathcal{K}_X,w_{02})$ yields
$$
\frac{w_{02}^1}{1+w_{02}^1} \ < \ \frac{l_{11}+l_{21}}{2l_{11}l_{21}} .
$$
Since the left side is at least $1/2$, we get $l_{11}=1$.
Now Remark~\ref{rem:posstrata} with $\gamma_{11,12,21,22}$
implies $l_{21}=w_{02}^1+1$.
Substituting these equalities in the inequality above we arrive at
$l_{21}<3$, therefore we have $l_{21}=2$
(for $l_{21}=1$ refer to situation~2C).
Taking $P$ into account, we use admissible operations
and equalities from $P\cdot Q^T=0$ and achieve
$$
P \ = \ 
\left[
\begin{array}{rrrrcr}
-1 & -1 & 1 & 1 & 0 & 0
\\
-1 & -1 & 0 & 0 & 2 & 1
\\
0 & 1 & d_{111} & 0 & -2d_{111}-1 & -1
\\
0 & 0 & d_{211} & 0 & -2d_{211} & 0
\end{array}
\right] ,
$$
where $0\le d_{111} < d_{211}$ holds.
In order to find an upper bound for $d_{211}$
we turn to the lineality part $A_{X,0}^c$ of the anticanonical complex.
Its vertices are
\begin{align*}
u_1 =& \ \frac{1}{2}( d_{111}-1,d_{211} ),
& u_2 =& \ u_1 \ + \ \left(\frac{1}{2},0\right), \\
u_3 =& \ \frac{1}{2}( -1,0 ),
& u_4 =& \ u_3 \ + \ \left(\frac{5}{6},0\right), \\
u_5 =& \ \frac{1}{3}( -2d_{111}-1,-2d_{211} ),
& u_6 =& \ u_5 \ + \ \left(\frac{2}{3},0\right).
\end{align*}
Consider $C:=\conv(u_2,u_3,u_4)\subset A_{X,0}^c$.
The point $u_3$ lies over the bisection of the first orthant
and, because of the terminality of $X$,
$C$ does not contain integral points. We conclude $d_{211}<20$.
With the MDSpackage~\cite{MDS} we find out that
$(d_{111},d_{211})$ assumes the value $(0,1)$ and delivers variety No.~5.

\medbreak
\emph{Situation}~2B: homogeneity of the relation
yields $l_{02}=l_{12}=w_{02}^2=1$ for the second component and
$l_{01}+w_{02}^1=l_{11}w_{11}^1=w_{21}^1$ for the first component.
With Remark~\ref{rem:posstrata} applied to $\gamma_{11,12,21,22}$
we conclude $w_{11}^1=1$, hence
$w_{21}^1=l_{21}$ and $w_{02}^1=l_{11}-l_{01}$.
We can discharge the possibility that
$-\mathcal{K}_X\in\cone(w_{02},w_{22})^\circ$ holds,
since in that case $\gamma_{02,12,22}$ is a relevant face
and $w_{02}^1=1$ follows, contradicting the fact that
the anticanonical class lies in that prescribed cone.
Thus $-\mathcal{K}_X\in\cone(w_{01},w_{02})^\circ$ holds.
In particular $\det(-\mathcal{K}_X,w_{02})>0$ holds,
which implies $l_{11}=l_{01}+1$.
We use admissible operations and $P\cdot Q^T=0$
and reach
$$
P \ = \ 
\left[
\begin{array}{rrrrrr}
-l_{01} & -1 & l_{01}+1 & 1 & 0 & 0
\\
-l_{01} & -1 & 0 & 0 & 1 & 1
\\
d_{101} & 0 & -d_{101} & -1 & 0 & 1
\\
d_{201} & 0 & -d_{201} & 0 & 0 & 0
\end{array}
\right] ,
$$
with $0\le d_{101} < d_{201}$.
We find bounds on $d_{201}$ and $l_{01}$
by considering the lineality part $A_{X,0}^c$,
whose vertices are
\begin{align*}
u_1 =& \ \frac{1}{l_{01}+1}( d_{101}-l_{01},d_{201} ),
& u_2 =& \ u_1 \ + \ \left(\frac{l_{01}}{l_{01}+1},0\right), \\
u_3 =& \ \frac{1}{2}( -1,0 ),
& u_4 =&  \ \frac{1}{2l_{01}+1}( l_{01}^2+d_{101}+l_{01} , d_{201} ) \\
u_5 =& \ \frac{1}{l_{01}+2}( -d_{101},-d_{201} ),
& u_6 =& \ u_5 \ + \ \left(\frac{l_{01}+1}{l_{01}+2},0\right).
\end{align*}
In particular, consider the width $l_C$ of
$C := \conv(0,u_1,u_2,u_4)$
at the height $h(u_4)$ of $u_4$, i.e.
$$
l_C \ = \ \frac{l_{01}(l_{01}+2)}{2l_{01}+1} . 
$$
Since $l_{01}>1$ (otherwise we are in configuration~A),
the width $l_C$ is greater than $1$ and as a consequence
$h(u_4)<1$ holds by terminality. This gives $d_{201}\le 2l_{01}$.
The length of the line segment $A_{X,0}^c\cap\{y=0\}$
increases when $l_{01}$ increases.
By terminality, it cannot be greater than 2,
hence we conclude $l_{01}<5$.
The MDSpackage~\cite{MDS} finds a lattice point in $A_X^c$
for each variety defined by such data, hence
this situation does not provide terminal varieties.

\medbreak
\emph{Situations}~2C~\emph{and}~2D:
homogeneity delivers
$w_{11}^1=w_{21}^1=l_{01}+l_{02}w_{02}^1$.
Hence, by Remark~\ref{rem:posstrata},
the relevant face $\gamma_{11,12,21,22}$ yields
$w_{11}^1=1$. Since all terms on the right side of the equation
are greater or equal to one, we reach a contradiction.

\medbreak
\emph{Disposition}~3: we start with the following degree matrix $Q$:
$$
Q \ = \ 
\left[
\begin{array}{cc|cc|cc}
w_{01}^1 & w_{02}^1 & w_{11}^1 & 0 & w_{21}^1 & 0
\\
w_{01}^2 & w_{02}^2 & 0 & w_{12}^2 & 0 & w_{22}^2
\end{array}
\right],
$$
where we could assume $\Eff(X)=\QQ^2_{\ge0}$
thanks to Remark~\ref{rem:posstrata} applied to $\gamma_{11,12,21,22}$.
We can assume that $-\mathcal{K}_X\in\cone(w_{01},w_{12})^\circ$ holds,
thus $\gamma_{01,12,22}$ is a relevant face and $w_{01}^1=1$ follows.
If $-\mathcal{K}_X\in\cone(w_{02},w_{12})^\circ$ holds,
then $\gamma_{02,11,21}$ is a relevant face and $w_{02}^2=1$ would hold,
contradicting $\det(w_{01},w_{02})>0$.
Hence $-\mathcal{K}_X\in\cone(w_{01},w_{02})^\circ$,
$\gamma_{02,12,22}$ is a relevant face and $w_{02}^1=1$ holds.

\medbreak
\emph{Situation}~3A: 
by homogeneity of the relation
$2=\mu^1=l_{11}w_{11}^1=l_{21}w_{21}^1$ holds.
At least one of the exponents is equal to two,
since $\gamma_{11,12,21,22}$ is relevant
and Remark~\ref{rem:posstrata} can be applied to it.
Let $l_{11}=2$, so $l_{21}\in\{1,2\}$.
Using admissible operations we can assume
$d_{101}=d_{201}=d_{202}=0$ and $d_{102}=1$.
Then Proposition~\ref{prop:combmin-char}
with $\{3,5\}$ and  $\{4,6\}$, together
with equations coming from $P\cdot Q^T=0$, delivers
$$
d_{121} \ = \ -\frac{1}{2}l_{21}(d_{111}+1), \qquad
d_{221} \ = \ -\frac{l_{21}d_{211}}{2}, \qquad
d_{222} \ = \ -\frac{l_{22}d_{212}}{l_{12}}.
$$
The vertices of the lineality part $A_{X,0}^c\subset\QQ^2$ are
\begin{align*}
u_1 =& \ \frac{l_{21}( 2d_{112}-l_{12}d_{111}-l_{12} , 2d_{212}-l_{12}d_{211} )}{2(l_{12}+l_{21})},
& u_2 =& \ u_1 \ + \ \left(\frac{l_{12}l_{21}}{l_{12}+l_{21}},0\right), \\
u_3 =& \ \frac{l_{21}}{l_{21}+2}( -1,0 ),
& u_4 =& \ u_3 \ + \ \left(\frac{2l_{21}}{l_{21}+2},0\right), \\
u_5 =& \ \frac{( l_{12}(2d_{122}+l_{22}d_{111}) , l_{22}(l_{12}d_{211}-2d_{212}) )}{l_{12}(l_{22}+2)},
& u_6 =& \ u_5 \ + \ \left(\frac{2l_{22}}{l_{22}+2},0\right).
\end{align*}
We go through both cases $l_{21}=1,2$.

First assume that $l_{21}=1$ holds.
Then we achieve $d_{121}=d_{221}=0$ by admissible operations,
and $P\cdot Q^T=0$ also yields $d_{111}=-1$.
In order for $u_3$ and $u_4$ to be both vertices,
$l_{12}=1$ must hold.
For $l_{22}>1$, the intersection of $\conv(0,u_5,u_6)$ with
the line $\{y=-1\}$ has length one, thus contains
an integral point and contradicts terminality,
whereas $l_{22}=1$ will be handled in situation~3B.

Now assume that $l_{21}=2$ holds.
By admissible operations we achieve
$0\le d_{111},d_{211} < 2$.
Since $v_{11}$ and $v_{21}$ are primitive,
we arrive at $d_{111}=0$ and $d_{211}=1$.
In order for $u_3$ and $u_4$ to be both vertices,
at least one between $l_{12}$ and $l_{22}$
is equal to one.
Since Remark~\ref{rem:posstrata} applies to
$\gamma_{11,12,21,22}$, they cannot be both equal to one.
Without loss of generality say $l_{12}=1$ and $l_{22}\ge2$.
Therefore $A_{X,0}^c\cap\{y=0\}$ has length one
and the length of the edge $\overline{u_5u_6}$ is at least one.
This means $\lvert u_5^2\rvert < 1$, i.e. $d_{212}=0,1$.
Using homogeneity in the second component and once again $P\cdot Q^T=0$
we arrive at
$$
Q \ = \ 
\left[
\begin{array}{cc|cc|cc}
1 & 1 & 1 & 0 & 1 & 0
\\
l_{22}+d_{122} & -d_{122} & 0 & l_{22} & 0 & 1
\end{array}
\right].
$$
In particular the anticanonical class is
$-\mathcal{K}_X=(2,l_{22}+1)$.
The inequalities coming from $\det(w_{02},-\mathcal{K}_X)>0$
and $\det(w_{01},w_{02})>0$ are incompatible with
$d_{122}$ being an integer:
$$
-\frac{1}{2}l_{22}-\frac{1}{2} \ < \
d_{122} \ < \ -\frac{1}{2}l_{22}.
$$
Therefore we reach a contradiction.

\medbreak
\emph{Situation}~3B:
homogeneity delivers $l_{01}+l_{02}=l_{11}w_{11}^1=w_{21}^1$
and $l_{12}w_{12}^2=w_{22}^2$.
With $\gamma_{11,12,21,22}$ relevant face we conclude
$w_{11}^1=1=w_{12}^2$.
Therefore the anticanonical class is
$-\mathcal{K}_X=(3,w_{01}^1+w_{02}^1+1)$.
Now $\det(w_{01},w_{02})>0$ and $\det(w_{02},-\mathcal{K}_X)>0$
yield $w_{02}^2<1$, a contradiction to the disposition.

\medbreak
\emph{Situations}~3C~\emph{and}~3D: the
same argument as situations~2C and~2D above works here too.

\medbreak
\emph{Disposition}~4:
without loss of generality we assume that
$-\mathcal{K}_X\in\cone(w_{01},w_{21})^\circ$ holds.
In particular $\gamma_{01,11,21}$ and $\gamma_{01,11,22}$
are relevant faces, to which Remark~\ref{rem:posstrata} applies;
we arrive at
$$
Q \ = \ 
\left[
\begin{array}{cc|cc|cc}
w_{01}^1 & 0 & w_{11}^1 & 0 & w_{21}^1 & w_{21}^1
\\
0 & w_{02}^2 & 0 & w_{12}^2 & 1 & 1
\end{array}
\right].
$$

\medbreak
\emph{Situation}~4A:
we use homogeneity of the relation and
Remark~\ref{rem:posstrata} together with relevant faces
$\gamma_{01,02,11,12}$ and $\gamma_{01,02,12,21,22}$
in the usual ways and achieve
$$
Q \ = \ 
\left[
\begin{array}{cc|cc|cc}
l_{11} & 0 & 1 & 0 & 1 & 1
\\
0 & l_{11} & 0 & 1 & 1 & 1
\end{array}
\right].
$$
Moreover we use admissible operations on $P$ to achieve
$d_{101}=d_{201}=d_{202}=0$ and $d_{201}=1$.
The equations of $P\cdot Q^T=0$ allow us to write
\begin{align*}
v_{12} \ &= \ (l_{12}, 0, d_{112}, d_{212} ) , \\
v_{11} \ &= \ v_{12} \ + \ (0,0,l_{21}+l_{22},0) .
\end{align*}
This means that there are integral points on the segment
between $v_{11}$ and $v_{12}$, contradicting terminality.

\medbreak
\emph{Situation}~4B:
we have $l_{02}w_{02}^2=l_{12}w_{12}^2=2$.
We may assume that one of those exponents is greater than one,
otherwise see situation~4E.
Without loss of generality let $l_{12}=2$ and $w_{12}^2=1$.
Using all the equations coming from $P\cdot Q^T=0$
we can write
$$
P \ = \ 
\left[
\begin{array}{ccrrrr}
-l_{01} & -l_{02} & l_{11} & 2 & 0 & 0
\\
-l_{01} & -l_{02} & 0 & 0 & 1 & 1
\\
-\frac{l_{01}d_{111}}{l_{11}}-\frac{1}{2}l_{01} & -\frac{1}{2}l_{02}(d_{112}+1) & d_{111} & d_{112} & 0 & 1
\\
-\frac{l_{01}d_{211}}{l_{11}} & -\frac{1}{2}l_{02}d_{212} & d_{211} & d_{212} & 0 & 0
\end{array}
\right] ,
$$
$$
Q \ = \ 
\left[
\begin{array}{cc|cc|cc}
\frac{2}{l_{01}}w_{21}^1 & 0 & \frac{2}{l_{11}}w_{21}^1 & 0 & w_{21}^1 & w_{21}^1
\\
0 & \frac{2}{l_{02}} & 0 & 1 & 1 & 1
\end{array}
\right].
$$
In particular we see that $l_{02}=1,2$ holds.

First we rule out the case $l_{02}=1$;
if it holds, we can reach $d_{112}=1$ and $d_{212}=0$
by means of admissible operations.
The vertices of $A_{X,0}^c$ are
\begin{align*}
u_1 =& \ \frac{1}{l_{11}+1}( d_{111}-l_{11},d_{211} ),
& u_2 =& \ u_1 \ + \ \left(\frac{l_{11}}{l_{11}+1},0\right), \\
u_3 =& \ \frac{1}{3}( -1,0 ),
& u_4 =&  \ u_3 \ + \ \left(\frac{2}{3},0\right), \\
u_5 =& \ \frac{2l_{01}}{(l_{01}+2)l_{11}}(-d_{111},-d_{211} ),
& u_6 =& \ u_5 \ + \ \left(\frac{2l_{01}}{l_{01}+2},0\right).
\end{align*}
The only way to ensure that $u_3$ and $u_4$ are vertices is to set $l_{11}=1$.
Hence the weights relative to the relevant face $\gamma_{02,11,21}$
are $(2,0)$, $(0,2)$ and $(1,1)$. These
points do not generate $\ZZ^2$ as a lattice. Therefore
the stratum $X(\gamma_{02,11,21})$ consists of singular points,
contradicting terminality by Remark~\ref{rem:posstrata}.

Now assume $l_{02}=2$.
Again, we look at the vertices of $A_{X,0}^c$ and
have to set $l_{11}=1$, after which we achieve
$d_{111}=d_{211}=0$ by admissible operations.
Moreover, since $v_{01}$ is primitive, $l_{01}=2$ holds.
We have
$$
P \ = \ 
\left[
\begin{array}{rcrrrr}
-2 & -2 & 1 & 2 & 0 & 0
\\
-2 & -2 & 0 & 0 & 1 & 1
\\
-1 & -1-d_{112} & 0 & d_{112} & 0 & 1
\\
0 & -d_{212} & 0 & d_{212} & 0 & 0
\end{array}
\right]
$$
and the vertices of the lineality part are
\begin{align*}
u_1 =& \ \frac{1}{2}( d_{112}-1,d_{212} ),
& u_2 =& \ u_1 \ + \ \left(1,0\right), \\
u_3 =& \ \frac{1}{2}( -1,0 ),
& u_4 =&  \ u_3 \ + \ \left(1,0\right), \\
u_5 =& \ \frac{1}{3}(-d_{112}-1,-d_{212} ),
& u_6 =& \ u_5 \ + \ \left(\frac{2}{3},0\right).
\end{align*}
Since $\conv(u_1,u_2,u_3,u_4)$
does not contain integral points other than the origin, we conclude
$d_{212}=1$ and, with an admissible operation, $d_{112}=0$.
These data define a valid variety, namely No.~6.

\medbreak
\emph{Situation}~4E:
homogeneity implies
$\mu^2 = w_{02}^2=w_{12}^2=l_{21}+l_{22}$.
Moreover, Remark~\ref{rem:posstrata} applied
on $\gamma_{01,02,11,12}$ prescribes $w_{02}^2=1$.
This means $l_{21}+l_{22}=1$, a contradiction.

\medbreak
\emph{Disposition}~5:
applying Remark~\ref{rem:posstrata}
to $\gamma_{02,11,21}$ and $\gamma_{02,11,22}$
we obtain the degree matrix
$$
Q \ = \ 
\left[
\begin{array}{cc|cc|cc}
w_{01}^1 & 0 & 1 & w_{12}^1 & w_{21}^1 & 0
\\
w_{01}^2 & 1 & 0 & w_{12}^2 & 0 & w_{22}^2
\end{array}
\right].
$$

\medbreak
\emph{Situation}~5A:
we divide this situation into three subcases.
They differ from one another by the Mori chamber $\alpha_i\subset\Eff(X)$
in which the anticanonical class $-\canK_X$ lies.
In all three cases we use admissible operations and
bring the defining matrix $P$ into the following shape:
$$
P \ = \ 
\left[
\begin{array}{rcrrrr}
-1 & -1 & l_{11} & l_{12} & 0 & 0
\\
-1 & -1 & 0 & 0 & l_{21} & l_{22}
\\
0 & 1 & d_{111} & d_{112} & d_{121} & d_{122}
\\
0 & 0 & d_{211} & d_{212} & d_{221} & d_{222}
\end{array}
\right] .
$$

\smallbreak
\emph{Situation}~5A~\emph{with $-\canK_X\in\alpha_1$}:
here $\gamma_{02,12,22}$ is a relevant face, hence
Remark~\ref{rem:posstrata} yields $w_{12}^1=1$.
Homogeneity of the relation implies
\begin{gather*}
w_{01}^1 \ = \ l_{11}+l_{12}, \qquad
w_{21}^1 \ = \ \frac{l_{11}+l_{12}}{l_{21}}, \\
w_{12}^2 \ = \ \frac{w_{01}^2+1}{l_{12}}, \qquad
w_{22}^2 \ = \ \frac{w_{01}^2+1}{l_{22}}.
\end{gather*}
Note that $\gamma_{01,02,21,22}$ is also a relevant face.
Hence $w_{21}^1=1$ holds, i.e. $l_{21}=l_{11}+l_{12}$.
Since the anticanonical class lies in $\alpha_1$,
we have $\det(w_{12},-\canK_X)>0$.
This implies $l_{12}>2l_{22}$, in particular $l_{12}\ge3$.
Through equations from $P \cdot Q^T=0$ we have
$$
d_{111}  =  -d_{112}-d_{121}, \qquad
d_{211}  =  -d_{212}-d_{221}, \qquad
d_{212}  =  -\frac{l_{12}d_{222}}{l_{22}}.
$$
Consider the vertices $u_1$ and $u_2$
of the lineality part $A_{X,0}^c$, defined by the elementary big cones
$\cone(v_{01},v_{1j},v_{21})$, for $j=1,2$ respectively.
The segment line $\overline{u_1u_2}$ intersects
the $x$-axis in the point $(0,0,(l_{11}+l_{12})/3,0)$.
Since $l_{12}\ge3$ holds, the lattice point $(0,0,1,0)$
lies in $\conv(0,u_1,u_2)\subset A_{X,0}^c$,
a contradiction to terminality.

\smallbreak
\emph{Situation}~5A~\emph{with $-\canK_X\in\alpha_2$}:
we use homogeneity of the relation and Remark~\ref{rem:posstrata}
on the relevant face $\gamma_{01,02,21,22}$ to arrive at
$$
Q \ = \ 
\left[
\begin{array}{cc|cc|cc}
l_{21} & 0 & 1 & (l_{21}-l_{11})/l_{12} & 1 & 0
\\
w_{01}^2 & 1 & 0 & (w_{01}^2+1)/l_{12} & 0 & (w_{01}^2+1)/l_{22}
\end{array}
\right].
$$
In particular $l_{21}\ge l_{11}+l_{12}$ holds.
The matrix $Q$ allows us to compute the anticanonical class $-\canK_X$
according to Remark~\ref{rem:fanoRAP}.
Since $-\canK_X\in\alpha_2$ holds, we have $\det(-\canK_X,w_{12})>0$.
Writing down this condition explicitly we obtain $l_{21}<l_{11}+2l_{22}$.
Now we turn to the matrix $P$.
Using equations from $P \cdot Q^T=0$ we determine
\begin{align*}
d_{111} \ &= \ -\frac{d_{112}(l_{21}-l_{11})}{l_{12}} - d_{121}, \\
d_{211} \ &= \ -\frac{d_{222}(l_{21}-l_{11})}{l_{22}} - d_{221}, \\
d_{212} \ &= \ -\frac{l_{12}d_{222}}{l_{22}}.
\end{align*}
Now consider the vertices $u_1$ and $u_2$
of the lineality part $A_{X,0}^c$, defined by the elementary big cones
$\cone(v_{02},v_{1j},v_{21})$, for $j=1,2$ respectively.
The segment line $\overline{u_1u_2}$ intersects
the $x$-axis in the point
$$
\Bigl(0,0,\frac{l_{12}l_{21}}{2l_{12}+l_{21}-l_{11}},0\Bigr).
$$
By terminality the lattice point $(0,0,1,0)$ does not lie
$\conv(0,u_1,u_2)\subset A_{X,0}^c$.
This implies $l_{11}=l_{12}=1$.
In this special situation, we achieve
$$
P \ = \ 
\left[
\begin{array}{rcrrrr}
-1 & -1 & 1 & 1 & 0 & 0
\\
-1 & -1 & 0 & 0 & l_{21} & l_{22}
\\
0 & 1 & 0 & d_{112} & -d_{112}(l_{21}-1) & -l_{22}d_{112}-1
\\
0 & 0 & 0 & d_{212} & -d_{212}(l_{21}-1) & -l_{22}d_{212}
\end{array}
\right] ,
$$
$$
Q \ = \ 
\left[
\begin{array}{cc|cc|cc}
l_{21} & 0 & 1 & l_{21}-1 & 1 & 0
\\
l_{22}-1 & 1 & 0 & l_{22} & 0 & 1
\end{array}
\right],
$$
where we can assume $0\le d_{112} < d_{212}$ and $l_{21},l_{22}\ge2$.
Consider the leaf $A_X^c \cap \lambda_2$ of the anticanonical
complex, embedded in $\QQ^3$ by removing the first coordinate 
(which always equals zero) from its points.
Define $B\subset\QQ^3$ as the convex hull of the following points
\begin{gather*}
b_1 := (l_{21},d_{121},d_{221}), \qquad 
b_2 := (l_{22},d_{122},d_{222}), \\
a_1 := (-1,0,0), \qquad a_2 := (-1,1,0), \\
a_3 := (-1,d_{112},d_{212}), \qquad
a_4 := (-1,d_{112}+1,d_{212}).
\end{gather*}
Then the leaf $A_X^c \cap \lambda_2$ corresponds to
the intersection $B\cap\{(x,y,z)\in\QQ^3; x\ge0\}$.
By terminality, the only integral points of the leaf are
$b_1$, $b_2$ and the origin.
Hence $B$ is a lattice polytope containing the origin as only interior point and,
by~cite{Ka:2010}, $\vol(B)$ is bounded by $12$.
This gives the condition
$$
d_{212}(l_{21}+l_{22}+2) \ < \ 36 .
$$
Therefore all entries of $P$ are bounded.
We use the MDSpackage~\cite{MDS} to check all possibilities.
It turns out that none of the matrices defines a terminal variety.

\smallbreak
\emph{Situation}~5A~\emph{with $-\canK_X\in\alpha_3$}:
here $\gamma_{01,11,21}$ is a relevant face, so
Remark~\ref{rem:posstrata} yields $w_{01}^2=1$.
Using homogeneity of the relation and $\gamma_{01,02,21,22}$
relevant face we arrive at
$$
Q \ = \ 
\left[
\begin{array}{cc|cc|cc}
l_{21} & 0 & 1 & (l_{21}-l_{11})/l_{12} & 1 & 0
\\
1 & 1 & 0 & 2/l_{12} & 0 & 2/l_{22}
\end{array}
\right].
$$
Since the anticanonical class lies in $\alpha_3$,
we have $\det(-\canK_X,w_{01})>0$.
This condition is equivalent to the inequality
$$
l_{11}l_{22}+2l_{12}l_{21}-2l_{12}l_{22}+l_{21}l_{22}<0.
$$
By looking at the matrix $Q$ we see that $l_{12},l_{22}\in\{1,2\}$ holds.
None of the possible combinations satisfies the condition above,
hence we reach a contradiction.

\medbreak
\emph{Situation}~5B: homogeneity implies
$w_{21}^1=l_{01}w_{01}^1$ and $w_{22}^2=l_{12}w_{12}^2$.
Using Remark~\ref{rem:posstrata}, respectively
with $\gamma_{01,02,21,22}$ and $\gamma_{11,12,21,22}$,
we obtain $w_{01}^1=1$ and $w_{12}^2=1$.
Since $w_{01}^2,w_{12}^1>0$ holds,
we arrive at a contradiction with the disposition of the weights,
because $\det(w_{01},w_{12})>0$ holds.

\medbreak
\emph{Situations}~5C~\emph{and}~5F: homogeneity
yields $w_{01}^1=w_{21}^1=l_{11}+l_{12}w_{12}^1$.
By Remark~\ref{rem:posstrata} with
$\gamma_{01,02,21,22}\in\rlv(X)$, we have $w_{01}^1=1$
but then $l_{11}+l_{12}w_{12}^1=1$ holds.
This is a contradiction, since all values
appearing on the left side are at least one.

\begin{proof}[Proof of Theorem~\ref{thm:combmin-list}]
Lemma~\ref{lem:terminal-combmin} lists all possible
constellations for the defining matrix $P$.
The results of the cases treated in this Section~\ref{sec:Qfact-classif},
combined with the remaining cases which are treated in~\cite[Chapter~3]{Nico},
deliver the list from the assertion.
Furthermore, by comparing the data, one directly sees that
any two varieties from the list are non-isomorphic.
\end{proof}

\begin{remark}
Appendix~A of~\cite{Nico} contains detailed information about the
$\QQ$-factorial varieties of Theorem~\ref{thm:combmin-list},
including possible defining matrices $P$.
\end{remark}

\begin{remark}
For $\KK = \CC$,
any Fano variety $X$ with at most log terminal singularities,
has finitely generated divisor class group $\Cl(X)$; see~\cite[Sec.~2.1]{IsPr:1999}.
If $X$ comes in addition with a torus action of complexity one,
then $X$ is rational and its Cox ring is finitely generated;  see~\cite[Remark~4.4.1.5]{ArDeHaLa}.
Therefore the assumption of rationality can be omitted
in Theorem~\ref{thm:combmin-list} for $\KK=\CC$.
Alternatively, rationality can be replaced by the property
``$\Cl(X)$ is finitely generated''.
\end{remark}


\section{Growing the anticanonical complex}
\label{sec:growanticancomp}

This last Section outlines the future road towards a complete classification
of terminal Fano threefolds of complexity one.

Let $X$ be a normal projective variety of complexity one.
The Cox ring $\Cox(X)$ is
$$
R \ := \ \Cox(X) \ = \ \KK[T_1,\ldots,T_r,T_{r+1}]/\!\braket{g_1,\ldots,g_s}
$$
with grading given by $K:=\Cl(X)$ and degrees $w_i := \deg(T_i)\in K$.
Recall that the weight $w_i$ is \emph{exceptional},
if $w_i \notin \cone(w_j; j\neq i)$ holds.

\begin{remark}
\label{rem:conesR}
Paraphrasing Remarks~\ref{rem:fanoRAP} and~\ref{rem:divcones} we say that
the \emph{moving cone} and the \emph{anticanonical class} of $R$ are respectively
\begin{align*}
\Mov(R) \ &:= \  \bigcap_{i=0}^{r+1} \cone(w_j ; j\neq i) \ \subset \ K_\QQ , \\
\kappa(R) \ &:= \  \sum_{i=1}^{r+1} w_i - \sum_{j=1}^s \deg(g_j) \ \in \ K.
\end{align*}
We say that $R$ is \emph{Fano} if $\kappa(R)\in\Mov(R)^\circ$ holds. 
The last definition is justified by the fact that 
$\kappa(R)\in\Mov(R)^\circ$ characterizes a ring $R$
that can be realized as Cox ring of a normal Fano variety,
which we then call $X_F(R)$.
\end{remark}

\begin{definition}
\label{defi:terminalRAP}
Let $R$ be Fano as in Remark~\ref{rem:conesR}.
We say that $R$ is \emph{terminal}
\emph{(resp. canonical, log-terminal, $\QQ$-factorial, etc...)}
if the Fano variety $X_F(R)$ has that property.
\end{definition}

\begin{construction}
\label{constr:contract}
Let $R$ and $K$ be as above. Assume that $w_{r+1}$ is exceptional.
Define polynomials $g_j' := g_j|_{T_{r+1}=1}$
and the ring $R' := \KK[T_1,\ldots,T_r]/\!\braket{g_1',\ldots,g_s'}$.
Its grading is given by $K' := K/\!\braket{w_{r+1}}$
via the canonical projection map $\pi := K\to K'$
and the degrees $\deg(T_i) := w_i' := \pi(w_i)$.
Then $R'$ is again the Cox ring of a normal projective variety of complexity one.
\end{construction}

\begin{remark}
\label{rem:contr-diagram}
We have the commutative diagram with exact rows
$$
\begin{tikzpicture}[scale=0.8]
  \node (A) at (-0.5,0) {$0$};
  \node (B) at (2,1) {$K_\QQ$};
  \node (C) at (2,-1) {$K'_\QQ$};
  \node (D) at (5,1) {$\QQ^{r+1}$};
  \node (E) at (5,-1) {$\QQ^r$};
  \node (F) at (7.5,0) {$\QQ^n$};
  \node (G) at (9.5,0) {$0$};
\draw[->] (B) to (A);
\draw[->] (C) to (A);
\draw[->] (B) to node[right] {{\tiny $\pi$}} (C);
\draw[->] (D) to node[above] {{\tiny $Q$}} (B);
\draw[->] (E) to node[below] {{\tiny $Q'$}} (C);
\draw[->] (D) to node[left] {{\tiny $\pi_r$}} (E);
\draw[->] (F) to node[above] {{\tiny $P^*$}} (D);
\draw[->] (F) to node[below] {{\tiny $P'^*$}} (E);
\draw[->] (G) to (F);
\end{tikzpicture}
$$
hence in particular $\pi \circ Q = Q' \circ \pi_r$ and $(P')^*=\pi_r \circ P^*$.
Therefore considering the presentation $R=R(A,P)$,
the ring $R'$ is given as $R'=R(A,P')$, where $P'$ is obtained from $P$
by deleting the column corresponding to the exceptional weight.
\end{remark}

\begin{proposition}
\label{prop:fano-contraction}
Let $R$ and $R'$ be as in Construction~\ref{constr:contract}.
If $R$ is (terminal) Fano, then so is $R'$.
\end{proposition}

\begin{proof}
To prove that the Fano property is preserved by Construction~\ref{constr:contract},
one only need to see that
$$
\kappa(R') \ = \ \pi(\kappa(R)) \ \in \ 
\pi(\Mov(R))^\circ \ \subseteq \ \Mov(R')^\circ .
$$
In order to study terminality, define $X$ and $X'$ as
the Fano varieties corresponding to $R$ and $R'$ respectively.
The goal is to show that $A_{X'}^c \subseteq A_X^c$ holds,
because then the assertion holds by Theorem~\ref{thm:BHHNmain}.
This can be done by considering the commutative diagram of Remark~\ref{rem:contr-diagram}
and proving that $B_{X'}\supseteq B_X$ holds.
For this one uses $\kappa(R') = \pi(\kappa(R))$ and $\pi_r(B(g_j))=B(g_j')$.
\end{proof}

\begin{remark}
Note that Proposition~\ref{prop:fano-contraction}
does not hold for $\QQ$-factoriality.
There are examples where $X(A,P)$ is $\QQ$-factorial
but $X(A,P')$ is not, and vice versa.
\end{remark}

\begin{corollary}
\label{cor:contractseries}
Let $X=X(A,P)$ be a terminal Fano threefold
of complexity one.
If $X$ is not combinatorially minimal,
then there is an equivariant small quasimodification $X \dasharrow X'$
onto a combinatorially terminal Fano threefold $X'$ of complexity one (which could be toric).
\end{corollary}

\begin{remark}
Much like in the case of toric Fano varieties with the growing Fano polytopes (see~\cite{Ka}),
Proposition~\ref{prop:fano-contraction} ensures that
we can approach the classification of terminal Fano varieties of complexity one
with purely combinatorial methods,
by taking the anticanonical complex of a combinatorially minimal variety and
successively adding new integral vertices, each time checking if all required properties still hold.
\end{remark}


\end{document}